\documentclass[final]{amsart}
\usepackage{amsmath, amsfonts, amssymb, amsthm}
\usepackage{mathtools}
\usepackage{lscape}
\usepackage{algorithm2e} \makeatletter \renewcommand{\@algocf@capt@plain}{above} \makeatother				% Place captions on top
\usepackage{hyperref}

\usepackage{graphicx}
\usepackage[table,xcdraw]{xcolor}
\usepackage[subrefformat=parens,labelformat=parens]{subfig}
 \makeatletter\let\@@magyar@captionfix\relax\makeatother 		% This prevent error message caused by 'sugfig' such as "undefined control \begin{document}"
\usepackage[export]{adjustbox}
\usepackage{caption}
\usepackage{pgfplots}
\usepackage{tikz}
\usetikzlibrary{shapes, arrows}
\usepackage{multimedia}
\usepackage{etaremune}
\usepackage{enumerate}
\usepackage{stmaryrd}
\usepackage{longtable}
\usepackage{tabularx}
\usepackage{wrapfig,kantlipsum}
\usepackage[normalem]{ulem} % for strike out
\usepackage{cancel} % for various cancel/modifying format
\usepackage{comment}

\newcommand{\norm}[2]{\bigl\| #1 \bigr\| _{#2}}
\newcommand{\pairing}[2]{\bigl\langle #1 , #2 \bigr\rangle}
\newcommand{\argmin}{\mathop{\mathrm{argmin}}}
			
\newcommand{\iprd}[2]{\left( #1 , #2 \right)}

\newcommand{\LL}{\mathcal{L}}
\newcommand{\LLinv}{{\mathcal{L}^{-1}}}
\newcommand{\lap}{\Delta}

\newcommand{\HH}{\mathbb{H}}

\newcommand{\RR}{\mathbb{R}}

\newcommand{\NN}{\mathbb{N}}
\newcommand{\ZZ}{\mathbb{Z}}

\theoremstyle{plain}
\newtheorem{thm}{Theorem}[section]
\newtheorem{prop}[thm]{Proposition}
\newtheorem{cor}[thm]{Corollary}
\newtheorem{lem}[thm]{Lemma}

\theoremstyle{definition}
\newtheorem{defn}[thm]{Definition}
\newtheorem{rmk}[thm]{Remark}

\newcommand{\ermk}{\hfill\ensuremath{\blacksquare}}

\newcommand{\goesto}{\rightarrow}
\newcommand{\half}{\frac{1}{2}}

\newcommand{\littleo}[1]{\operatorname{o}\bigl( #1 \bigr)}
\newcommand{\bigo}[1]{\operatorname{O} \bigl( #1 \bigr)}
\newcommand{\oneover}[1]{\frac{1}{#1}}

\newcommand{\emb}{\hookrightarrow}
\usepackage[numbers]{natbib}
\newcommand{\per}{{\textup{per}}}
\newcommand{\diff}{{\textup{d}}}
\newcommand{\fraki}{{\mathfrak{i}}}
\newcommand{\sz}{{\sqrt{s}}}
\newcommand{\x}{{\mathbf{x}}}

\newcommand{\gridfn}{\mathbb{H}_N }

\newcommand{\bfm}{{\mathbf{m}}}
\newcommand{\bfr}{{\mathbf{r}}}
\newcommand{\bfs}{{\mathbf{s}}}
\newcommand{\gm}{{\Omega_N}} % Grid Domain
\newcommand{\muhat}{{\hat \mu}}
\newcommand{\Lhat}{{\hat L}}

\DeclareMathOperator{\identity}{id}

\numberwithin{equation}{section}

\begin{document}
\bibliographystyle{plainnat}
	
\title[Preconditioned accelerated gradient descent]{Preconditioned accelerated gradient descent methods for locally Lipschitz smooth objectives with applications to the solution of nonlinear PDEs}

\author{Jea-Hyun Park}
\email[J.-H. Park]{jpark79@vols.utk.edu}

\author{Abner J. Salgado}
\email[A.J. Salgado]{asalgad1@utk.edu}

\author{Steven M. Wise}
\email[S.M. Wise]{ swise1@utk.edu}

\address{Department of Mathematics, The University of Tennessee, Knoxville, TN 37996, USA} 

	\begin{abstract}
	We develop a theoretical foundation for the application of Nesterov's accelerated gradient descent method (AGD) to the approximation of solutions of a wide class of partial differential equations (PDEs).
% 	where the dimension of the problem tends to infinity unlike typical optimization problems. 
    This is achieved by proving the existence of an \emph{invariant set} and exponential convergence rates when its preconditioned version (PAGD) is applied to minimize \emph{locally Lipschitz} smooth, strongly convex objective functionals.
%     To facilitate our analysis, 
    We introduce a second-order ordinary differential equation (ODE) with a preconditioner built-in and show that PAGD is an explicit time-discretization of this ODE, which requires a natural time step restriction for energy stability. At the continuous time level, we show an exponential convergence of the ODE solution to its steady state using a simple energy argument.  At the discrete level, assuming the aforementioned step size restriction, the existence of an invariant set is proved and a matching exponential rate of convergence of the PAGD scheme is derived by mimicking the energy argument and the convergence at the continuous level. Applications of the PAGD method to numerical PDEs are demonstrated with certain nonlinear elliptic PDEs using pseudo-spectral methods for spatial discretization, and several numerical experiments are conducted. The results confirm the global geometric and mesh size-independent convergence of the PAGD method, with an accelerated rate that is improved over the preconditioned gradient descent (PGD) method.
	\end{abstract}


\subjclass[2010]{65B99, % Acceleration in Numerical Analysis-none of these
65J08, % 	Numerical solutions to abstract evolution equations
65N35, %  	Spectral, collocation and related methods for boundary value problems involving PDEs
65K10%  	Numerical optimization and variational techniques
}
\maketitle

	\section{Introduction}

The purpose of this work is to broaden the context in which a  well-known and efficient algorithm for unconstrained convex minimization, the so-called \emph{Nesterov's accelerated gradient descent} (AGD) scheme, can be utilized and, further, to  shed some light on its convergence properties. This method iteratively finds approximations to the solution of the following optimization problem: given $G : \HH \to \RR$, find
	\[
x^* = \argmin \left\{ G(x) \, \middle| \, x \in \HH \right\}.
	\]
Here, and in what follows, $\HH$ is a real, separable Hilbert space with inner product $(\, \cdot\, ,\, \cdot\, )_\HH$ and the so-called \emph{objective} functional, $G$, is assumed to be strongly convex and \emph{locally Lipschitz} smooth; see Section~\ref{sec:prelim} for definitions and notation. We immediately comment that the assumptions on the objective guarantee the existence and uniqueness of a minimizer (e.g., \citep[Theorem 7.4-4, Theorem 8.2-2]{ciarlet89}). 

Convex minimization is ubiquitous, and our main interest in this problem comes from the fact that many important nonlinear partial differential equations (PDEs) can be viewed as the Euler equations of certain convex objective functions. 
For example, the classical minimal surface problem (see \citep{evans2010partial}) and the $p$--Laplacian equation (see \citep{barrett93}) have this structure, just to name a few. However, the method discussed in this work is particularly powerful for semilinear PDEs. For example, time discretizations of many important models in material science often end up involving such problems, e.g., via the convex splitting technique (see \citep{Wise2018FCH}). In a related context, the current explosion of interest in statistical learning has drawn the attention of practitioners to so-called first order schemes, i.e., those that only require knowledge of first order derivatives, which is suitable in dealing with large data sets. These considerations are important for solving nonlinear PDE as well. One of the main thrusts of this research is to show that Nesterov's accelerated schemes, which are popular in statistical learning, can be utilized as fast solvers for nonlinear PDE once we resolve some nontrivial, technical difficulties specific to such problems.

The first and most na\"ive approach to find $x^*$ would be to appeal directly to the first order necessary (and, in this context, sufficient) optimality condition, namely, the Euler equation:
	\begin{equation}
	\label{eq:FOcondition}
G'(x^*) = 0,
	\end{equation}
where $G'$ denotes the Fr\'echet derivative of $G$. In the examples that we have in mind, however, this requires the simultaneous solution to a very large number of nonlinear equations, and the direct solution of the system is not feasible in practice. Other approaches better suited for minimization must be constructed. According to \citep{bertsekas1999nonlinear}, iterative methods for minimizing functionals date back in 1847 when Cauchy proposed the so-called gradient descent method (GD). The solution to \eqref{eq:FOcondition} can be seen as the steady state of the \emph{gradient flow}
\begin{align}
\label{eq:GradFlow}
  X(0)=x_0, \qquad \qquad \dot X(t)=-G'(X(t)), \ t>0.
\end{align} 
Here $x_0 \in \HH$ is arbitrary and, in the second equation, we are implicitly identifying the dual space of $\HH$, denoted by $\HH'$, with $\HH$ itself. Under the assumptions we have imposed on the objective $G$, it is possible to show that this flow satisfies $X(t) \to x^*$ as $t\to \infty$, see \citep[Theorem 2.4]{wibisono2016}. The idea of GD is to approximate the solution to this flow via a \emph{forward Euler} time discretization with a fixed step size $s$: given $x_0 \in \HH$, for $k \geq 0$, find $x_{k+1}$ satisfying
	\begin{equation}
	\label{eq:GD}
x_{k+1} = x_k - sG'(x_k).
	\end{equation}
While this idea seems straightforward, more in-depth discussions on this method started only in the 1960s, where some practical step size rules and convergence analyses were established. It was shown that if the objective functional is convex  and Lipschitz smooth, then GD converges to the minimizer, $x^*$, and it exhibits a first order rate of convergence in the objective. Here, and in what follows, by an \emph{$n^{\rm th}$ order (algebraic) convergence} in the objective, we mean that $G(x_k)-G^*\le \bigo{1/k^n}$, as $k\goesto \infty$, where $G^*=G(x^*)$ is the minimum of $G$. By an \emph{exponential} or a \emph{geometric convergence} in the objective we mean that $G(x_k)-G^*\le \bigo{r^k}$, as $k\goesto \infty$, for some $r\in (0,1)$. In the latter case, we call $r$ the rate of (exponential) convergence. It can further be shown that, if the objective is, in addition, strongly convex, then the rate of convergence is exponential, and that it matches the rate of convergence of the solution of \eqref{eq:GradFlow} to $x^*$. (See \citep[Theorem 2.1.15]{nesterov_2014} or Remark~\ref{rmk:rate-match}). Some physical intuition for the evolution of the solution to \eqref{eq:GradFlow} is provided in Section \ref{sec:ODE-model}. See, in particular, Remark \ref{rmk:ODE-phys-interp}.

Evidently, all considerations regarding convergence are subject to the norm $\norm{\, \cdot \, }{\HH}$. It is possible to improve the convergence rate by using an equivalent norm, through which the level sets of the objective $G$ look ``more circular.''  In the numerical linear algebra and numerical PDE communities, this is commonly known as \emph{preconditioning}. In the context of \eqref{eq:GradFlow} and GD, this is achieved by introducing an operator $\LL:\HH\goesto \HH'$ and considering the evolution of $\dot X(t)=-\LLinv G'(X(t))$. Notice that we no longer implicitly identify $\HH'$ with $\HH$. The time-discrete counterpart of \eqref{eq:GD} is known as the \emph{preconditioned gradient descent method} (PGD) and is as follows: given $x_0 \in \HH$, for $k \geq 0$, find $x_{k+1}$ such that
\[
x_{k+1}=x_k-s \LLinv G'(x_k).
\]
If the preconditioner is suitably chosen, then the convergence rate of GD can be substantially improved (see \citep{feng2017preconditioned}). Note that we will tacitly assume in the sequel that $\mathcal{L}$ is independent of the iteration index $k$. 
We remark that Newton's method may be viewed as a kind of \emph{generalized} preconditioned gradient descent method if we assume that $G$ is twice Fr\'{e}chet differentiable and allow for the possibility that the preconditioner can change at each iteration. In particular, Newton's method is expressed as
\[
G''(x_k)\left(x_{k+1} -x_k\right) = - G'(x_k)  =: r_k ,	
\]
where $G''(x_k)$ is the second Fr\'{e}chet derivative of $G$, and $r_k$ is the so-called residual. Then, Newton's method is a generalized preconditioned gradient descent method for which the preconditioner satisfies $s\mathcal{L}_k = G''(x_k)$. One of the difficulties with Newton's method is that the preconditioner constantly changes, in general, and must be recomputed and re-inverted at each iteration step, which can prove quite costly. Furthermore, $G''$ may not exist in all applications of interest. Indeed, in the sequel, we will not assume that $G''$ exists.

To improve the convergence rate of GD, \citet{nesterov1983} suggested a scheme that accelerates the GD method. For convex and Lipschitz smooth objectives the Nesterov's accelerated gradient descent (AGD) scheme achieves a second order convergence rate. Later, he showed that if the objective is, in addition, strongly convex, then AGD achieves a faster exponential convergence rate than GD (see \citep[Theorem 2.1.15]{nesterov_2014}).

However, while the GD scheme has a strong physical intuition behind it, it is not completely clear what mechanism is at play to provide an acceleration in the AGD scheme. Some attempts have been made to understand this in the literature.
	\citet{Attouch2000} studied asymptotic behaviors of the solutions to a heavy ball system (similar to \eqref{the-IVP}) and showed their convergence to minimizers (if they exist) of locally Lipschitz objectives that are bounded below at the continuous time level.  
\citet{Goudou2009} looked into a similar system with quasiconvex, locally Lipschitz objectives at the continuous time level and the convergence of its implicit discretization, what they call \emph{proximal inertial algorithm}.  
	Apparently, \citep{su} is the first work that explains the acceleration happening in AGD both quantitatively and intuitively and inspired many researchers including us. For convex, Lipschitz smooth objectives, they were able to show that the solutions to a second order ODE $\ddot X +\frac{3}{t} \dot X +\nabla G(X)=0$ converges to the set of minimizers of $G$ quadratically fast, as $t\to \infty$, and a matching, discrete convergence rate was established for a version of AGD.  
 \citet{wibisono2016} took a similar approach in more generality in the language of \emph{Bregman Lagrangian flow}. However, these two works did not explain the exponential acceleration for the strongly convex objectives. This limitation was one of the motivations of our work and removing it is one of the goals of this paper.  
	Recently, there have appeared more works that address the same issue and provide more general or unifying frameworks. For (globally) $L$--smooth, $\mu$--stronlgy convex objectives (see Section~\ref{sec:prelim} for definitions and notation),  the best known convergence rates of AGD and the associated ODE are $G(x_k)-G^*\le \bigo{{(1-\sqrt{\mu/L})^k}}$ as $k\to\infty$ and $G(X(t))-G^*\le \bigo{e^{-\sqrt \mu t}}$ as $t\to\infty$ respectively. Similarly, for (globally) $L$--smooth, convex objectives, $G(x_k)-G^*\le \bigo{1/k^2}$ as $k\to\infty$ and $G(X(t))-G^*\le \bigo{1/t^2}$ as $t\to\infty$ are the best known convergence rates respectively.  
Within the same framework as in \citep{wibisono2016}, but using a different Lyapunov function, \citet{Wilson2018} showed the best convergence rates for both convex and strongly convex cases at both continuous and discrete level.  
	\citet{Shi2018} looked into what they call \emph{high-resolution ODE} and provided a finer understanding about the momentum-type schemes and obtained similar results. The convergence rates for the strongly convex case that they derived were not the best. However, they were able to explain the difference in the performance of Polyak's momentum method and AGD at the continuous level.
 \citet{Siegel2019} also analyzed a system of ODEs to study a version of AGD and obtained the best known rate of convergence for the strongly convex case at both continuous and discrete levels. He also studied non-smooth but still strongly convex objectives and stochastic versions.  
	 \citet{Luo2019} obtained the same best convergence rates for all the four cases mentioned above using a single ODE system but using a time rescaling argument when dealing with the convex case.
\citet{Laborde2020} studied perturbed ODE systems and the corresponding version of AGD in the stochastic framework. As a byproduct, they obtained the same best convergence rates for the perturbed version of AGD with strongly convex objectives in the deterministic setting.
	There are also related works from a PDE point of view rather than convex optimization. \citet{Schaeffer2016} studied accelerated methods for nonlinear elliptic operators, which may not have a variational structure, i.e., the PDE may not have an appropriate objective. They also proposed similar methods for viscosity solutions.
\citet{Benyamin2018} and  \citet{Calder2019} studied \emph{PDE accelerations} that are similar to \citep{wibisono2016} in spirit and applied them to image processing and minimal surface obstacle problems respectively.

The work contained herein includes the following important contributions, in particular, from a numerical PDE point of view, which can be seen more clearly from the literature comparison of Appendix~\ref{sec:FancyTable}.  See Table \ref{tab:literature}.  % to our understanding of PAGD.
	\begin{enumerate}[1.]
 \item
 We prove all of our results under the more general assumption that the objective functional is \emph{locally Lipschitz} smooth. Almost all earlier works assume that the objective is \emph{globally Lipschitz} smooth (see e.g., \citep{nesterov_2014, wibisono2016, su, allen2014linear,Wilson2018,Laborde2020,Schaeffer2016,Luo2019,Siegel2019,Shi2018}). This is too restrictive to approximate solutions of nonlinear PDEs. If the objective functional associated with the PDE of interest grows just a bit faster than quadratic functionals (i.e., those of very mild nonlinearity), it violates the global Lipschitz condition and is beyond the theoretical guarantee. On the other hand, the local condition does not require anything outside of a certain bounded set so that much more nonlinear PDEs can be dealt with. Only a few works from a  dynamical system point of view (e.g., \citep{Attouch2000,Goudou2009}) assume local Lipschitz condition. However, those works address only convergence itself at the continuous time level and did not discuss discrete level analysis.  To the best of our knowledge, this is the first work that provides convergence rates under the local Lipschitz smoothness assumption at the continuous or discrete level.
 
 \item
 We prove the existence of an invariant set $\mathfrak B$ of the PAGD method. That is, every sequence generated by the scheme stays in a certain bounded set.  The local Lipschitz assumption is meaningful when it is furnished with an invariant set so that we have no restriction in exploiting the Lipschitz condition. We emphasize that this is not a trivial technicality. Unlike the gradient descent method, the accelerated methods are not descent methods. In fact, they oscillate. Thus, a simple sublevel set argument does not work. Even worse, they involve extrapolations of the main iterates. Consequently, a na\"ive attempt to obtain an invariant set leads to an impasse:
%  egg-first-chicken-first scenario: 
 to control the extrapolations, one wants to use the Lipschitz condition, but under the local Lipschitz condition, one cannot use it before proving that they are in a fixed bounded set. Again, to the best of our knowledge, our work is the first that addresses and resolves this issue. 
 
 \item
We provide a detailed discrete analysis for a nonlinear PDE.  All existing works mentioned before either did not discuss numerical examples or did not explain how concrete numerical examples fit the abstract framework, and they omitted whether their numerical examples satisfy the assumptions that they imposed.  In contrast, we show that our examples satisfy all the necessary assumptions.

 \item 
 We provide an intuitive explanation for the acceleration mechanism behind AGD for strongly convex and locally Lipschitz smooth objectives. Inspired by \citep{su}, we view AGD as a discretization of a certain second order ordinary differential equation (ODE) ---  we present how to discretize this ODE to obtain AGD ---  and show that the solution to this ODE converges exponentially fast to its stationary point, which is the minimizer of $G$. 
We also provide an energy based proof of the exponential convergence rate of PAGD. This proof mimics the analysis of the continuous counterpart that is previously developed and shows what dissipation mechanisms are at play to achieve the aforementioned acceleration. We also show that the rates of convergence of the ODE model and AGD match and those rates at the continuous and discrete level are both the best known rates.

  	\item
 We build a preconditioner into the problem itself (even at the continuous time level) to analyze the scheme with a preconditioner in an explicit way. This seems deceptively simple. After all, preconditioning is nothing but using a different norm, hence a numerical analysis in one norm implicitly suggests the possibility of a similar analysis in another norm. On the other hand, determining an effective preconditioning strategy is a nontrivial matter. Furthermore, it is our observation that preconditioning in the classical optimization setting, especially for problems related to data analysis and machine learning, is underutilized and is a potential growth area in the future. Likewise, preconditioning strategies related to spectral collocation methods applied to nonlinear PDE are uncommon, but, as we shall see, are effective and efficient solver tools.

	\end{enumerate}

This paper is organized as follows. In Section \ref{sec:prelim}, we summarize the notation, assumptions, and main tools that we will use. In Section \ref{sec:algorithms}, we introduce several numerical schemes that are closely related to our discussion and summarize their convergence rates. In Section \ref{sec:ODE-model}, we explore the connection between PAGD and a second order ODE and how this connection can help understand the acceleration behind PAGD intuitively. In Section \ref{sec:PAGD-conv}, we prove the existence of an invariant set for the PAGD scheme and its exponential convergence. We take an ODE inspired approach, whose intuition lies in the developments of Section \ref{sec:ODE-model}. In Section \ref{sec:numerical}, we illustrate the application of the PAGD method to the solution of some numerical PDEs. These numerical experiments show the improvement in convergence by both acceleration and preconditioning. Finally, in the Appendices, %\ref{app-PAGD-to-IVP}, 
we provide the derivation of the initial value problem (IVP) which corresponds to the limiting case of PAGD and a specific discretization of the IVP that leads to PAGD.

	\section{Preliminaries}
	\label{sec:prelim}

Let us begin by introducing the setting, assumptions, and some basic properties of the objects that we are interested in. By $\HH$, we denote a real and separable Hilbert space with inner product $(\, \cdot\, ,\, \cdot\, )_\HH$ and associated norm $\norm{\, \cdot \, }{\HH}$. Since we will use other inner products and norms on $\HH$, for clarity, we will refer to $(\, \cdot\, ,\, \cdot\, )_\HH$ and $\norm{\, \cdot \, }{\HH}$ as the \emph{canonical inner product} and \emph{canonical norm}, respectively. The dual of $\HH$ is denoted by $\HH'$. Its canonical operator norm is denoted by $\norm{\, \cdot \, }{\HH'}$. For $v\in\HH$ and $f\in\HH'$, the symbol $\pairing{f}{v}$ represents their duality pairing, that is, $\pairing{f}{v}=f(v)\in\RR$.

A \emph{preconditioner} is defined by a linear operator $\LL:\HH\goesto\HH'$. Such an operator induces a bilinear form: for $x,y\in\HH$, 
	\begin{equation}
(x,y)_\LL = \pairing{\LL x}{y} =\LL[x](y) .
	\label{def-L-iprd}	
	\end{equation} 
We further assume that the bilinear form defined in \eqref{def-L-iprd} satisfies the following properties: there exist $C_1, C_2>0$ such that, for any $x,y\in\HH$,
	\begin{align}
(x,y)_\LL =(y,x)_\LL ,
\quad 
(x,y)_\LL \le C_2 \norm{x}{\HH} \norm{y}{\HH}, 
\quad 
C_1 \norm{x}{\HH}^2 \le (x,x)_\LL	. 
\label{L-iprd-coerc}
	\end{align}

Let us state some immediate, but important consequences without proof for the sake of brevity.

	\begin{prop}[properties of $\LL$]
Let $\HH$ be a real, separable Hilbert space with inner product $(\, \cdot\, ,\, \cdot\, )_\HH$, and suppose that $\LL:\HH\goesto\HH'$ is a linear mapping that satisfies \eqref{L-iprd-coerc}.  Then, $( \, \cdot \, , \, \cdot \,)_\LL$ is an inner product on $\HH$ and the object
	\[
\norm{x}{\LL}=\sqrt{(x,x)_\LL}, \qquad \forall \, x \in \HH,
	\]
is a norm, which is, in fact, equivalent to the canonical norm, $\norm{\, \cdot \,}{\HH}$. By the Riesz Representation Theorem, $\LL$ is invertible. The inverse is continuous and, in fact, it is just the Riesz Map with respect to the $\LL$--inner product, denoted $\mathfrak{R}_{\LL}$. We write $\LLinv = \mathfrak{R}_{\LL}:\HH'\to \HH$. The object
	\begin{equation}
  (f,g)_{\LLinv}=\pairing{f}{\LLinv g}, \quad \forall \, f,g\in \HH',
\label{def-Linv-inner-prod}
	\end{equation}
is an inner product on the Hilbert space $\HH'$ and the object
	\begin{equation}
\norm{f}{\LLinv}=\sqrt{(f,f)_{\LLinv}}=\sqrt{\pairing{f}{\LLinv f}}, \quad \forall \, f\in \HH',
\label{def-Linv-norm}
	\end{equation}
is a norm. The new norm on $\HH'$ is an operator norm in the sense that 
	\begin{equation}
\norm{f}{\LLinv}=\sup_{0\ne x\in\HH} \frac{\pairing{f}{x}}{\norm{x}{\LL}} = \sup_{\substack{ x\in\HH \\ \|x\|_\LL = 1}} \pairing{f}{x} , \quad \forall \, f\in \HH'.
	\label{L*-Linv-eq}
	\end{equation}
Finally, we have
	\begin{equation}
\norm{\LLinv f}{\LL}=\norm{f}{\LLinv}, \ \forall \, f\in \HH' \qquad \mbox{and} \qquad \|\LL x\|_\LLinv=\|x\|_\LL, \ \forall \, x\in \HH .
	\label{L-Linv-rel}	
	\end{equation} 
	\end{prop}

	\begin{rmk}[no preconditioning]
By setting $\mathcal{L} = \mathfrak{R}_{\HH}^{-1}$, we can remove the preconditioning, where $\mathfrak{R}_{\HH}:\HH'\goesto\HH$ is the canonical Riesz map. Hence, PAGD is a generalization of AGD. \ermk

\end{rmk}

Our objective $G:\HH \to \RR$ will be assumed to be Fr\'echet differentiable at every point in $\HH$. We denote by $G'(x)\in\HH'$ the Fr\'echet derivative of $G$ at the point $x\in\HH$. Since the definition of Fr\'echet differentiability involves a norm, the actual derivative is possibly norm dependent. The following result shows that, actually, the definition is invariant as long as the  norms are equivalent.

\begin{prop}[equivalent norms]
Let $\HH$ be a real and separable Hilbert space with norm $\norm{\, \cdot \, }{\HH}$, and $G:\HH\goesto\RR$ be Fr\'echet differentiable at $x \in \HH$. Assume that $\interleave\cdot\interleave_\HH$ is another norm on $\HH$. If $\interleave \cdot \interleave_\HH$ is equivalent to $\norm{\, \cdot \, }{\HH}$, then $G$ is also  Fr\'echet differentiable at $x$ with respect to $\interleave\cdot\interleave_\HH$. Furthermore, the derivatives coincide.
\end{prop}

Notice that nothing is said about continuity in the previous statement. For convex functions, the continuity of the derivatives is automatic once the Fr\'echet differentiability is guaranteed (see \citep[p. 20 Corollary]{Phelps1993convexfn}).

\begin{prop}[continuity]
Let $\HH$ be a real and separable Hilbert space and $D \subset \HH$ be open and convex. If $G:D\goesto\RR$ is convex and Fr\'echet differentiable, then $x\mapsto G'(x)$ is norm continuous on $D$. 
\end{prop}

The following two definitions provide a framework to describe the geometry of the graph of our objective functional. 

	\begin{defn}[Lipschitz smoothness]
Let $\HH$ be a real and separable Hilbert space, and $G:\HH\goesto\RR$ be Fr\'echet differentiable at every point. We say that $G$ is \emph{locally Lipschitz smooth} (with respect to $\LL$--norm) iff, for every bounded, convex set $B \subset \HH$, there exists a constant $L_B>0$ such that
	\begin{equation}
\pairing{G'(x)-G'(y)}{x-y} \le L_B\norm{y-x}{\LL}^2 \quad \forall \, x,y\in B.
	\label{def:L-smooth}
	\end{equation}
For brevity, we say that $G$ is \emph{$L_B$--smooth on $B$}. If the constant $L_B=L>0$ can be chosen to be independent of $B$, then we say that $G$ is \emph{globally Lipschitz smooth with a constant $L$}, or simply \emph{$L$--smooth}.
	\end{defn}

\begin{rmk}[terminology]
The above definition is a weaker notion than the local Lipschitz continuity of the Fr\'echet derivative of $G$, which is given by
\begin{equation}
  \norm{G'(x)-G'(y)}{\LLinv}\le L_B\norm{x-y}{\LL} \quad \forall x,y\in B,
\label{def:Lip-str}
\end{equation}
for some $L_B>0$. Of course, this implies the local Lipschitz smoothness of $G$ \eqref{def:L-smooth}. In this paper, to avoid confusion, whenever \eqref{def:Lip-str} holds, we will say that $G$ is \emph{locally Lipschitz smooth in the strong sense} or that \emph{$G'$ is locally Lipschitz in the strong sense}. We need this stronger condition when we conduct the continuous level analysis (Section \ref{sec:ODE-model}). Note, however, for convex functions, the global versions of the two definitions are equivalent. That is, if $B=\HH$, \eqref{def:L-smooth} implies \eqref{def:Lip-str} (see \citep[Theorem 2.1.5 (2.1.8)]{nesterov_2014}). 
\ermk
\end{rmk}

	\begin{defn}[strong convexity]
Let $G:\HH\goesto\RR$ be Fr\'echet differentiable. We say that $G$ is \emph{$\mu$--strongly convex} (with respect to $\LL$--norm) iff there exists a constant $\mu>0$ such that
	\begin{equation}
\pairing{G'(x)-G'(y)}{x-y}\ge \mu	 \norm{y-x}{\LL}^2 \quad \forall \, x,y\in\HH .
	\label{strconv-coerc-prec}
	\end{equation}
	\end{defn}

We now state an equivalent characterization of these notions.

	\begin{thm}[equivalence]
	\label{equi-precon}
Let $\HH$ be a real and separable Hilbert space, and $G: \HH \goesto \mathbb{R}$ be Fr\'echet differentiable. $G$ is $L_B$--smooth on the bounded convex set $B \subset \HH$ if and only if
	\begin{equation}
  G(y)-G(x)-\pairing{G'(x)}{y-x} \le \frac{L_B}{2} \norm{y-x}{\LL}^2 \quad \forall \, x,y\in B.
	\label{LipDer-quad-pre}
	\end{equation}
Similarly, $G$ is $\mu$--strongly convex if and only if
	\begin{equation}
  G(y)-G(x)-\pairing{G'(x)}{y-x} \ge \frac{\mu}{2} \norm{y-x}{\LL}^2 \quad \forall \, x,y\in\HH .	
	\label{def:str-convex}
	\end{equation}
	\end{thm}
	\begin{proof}
These results follow from Taylor's Theorem with integral remainder. See also \citep[Theorem 2.1.5, Theorem 2.1.9]{nesterov_2014}.
	\end{proof}

Among the two characterizations of Lipschitz smoothness and strong convexity stated above, we will call \eqref{LipDer-quad-pre} and \eqref{def:str-convex} the upper and the lower \emph{quadratic trap} of $G$, respectively. 
The constant $\tfrac{L_B}{\mu}$ is called the \emph{(local) condition number} of the objective functional $G$ with respect to the $\LL$--norm. In what follows, we will use its reciprocal, denoted by $\rho=\tfrac{\mu}{L_B} \in (0,1]$, to quantify rates of convergence. Note that the condition number crucially depends on the norm that is used to describe the geometry of the graph of $G$. Choosing a good preconditioner, $\LL$, is at the heart of much of scientific computing.

We conclude this section by stating a pair of well-known identities which we will use frequently. For any $A,B\in\HH$, 
	\begin{align}
	(A,B)_\LL&=\half \norm{A}{\LL}^2 +\half \norm{B}{\LL}^2-\half \norm{A-B}{\LL}^2
	\label{3pt-id-}
	\\
	&=\half \norm{A+B}{\LL}^2 -\half \norm{A}{\LL}^2-\half \norm{B}{\LL}^2
	\label{3pt-id+}.
	\end{align}

\section{Optimization schemes}
\label{sec:algorithms}

Here we briefly review several algorithms that are closely related to our main algorithm of interest. To focus on the main differences between the schemes of interest, we will not pay attention to choices of step size and stopping criteria of the algorithms. For those readers who are interested in these details, we refer, for instance, to \citep{nesterov_2014, boyd2004convex, beck2017ch8, Chen2018Convergence}.

\begin{algorithm}
\caption{Preconditioned gradient descent method (PGD)}

  \KwData{$G$: The objective}
  \KwData{$s>0$: The step size}
  \KwData{$x_0 \in \HH$: The initial guess}
  \KwResult{The sequence $\{x_k\}_{k\geq1}$ that approximates $x^*$, the minimizer of $G$}
  \For{$k\geq 0$}{
    $ x_{k+1} = x_k - s\LLinv G'(x_k)  $\;
  }
\label{alg:PGD}
\end{algorithm}

We begin by presenting the PGD scheme in Algorithm~\ref{alg:PGD} and describing its convergence properties. To do so, we introduce
	\[
B = \left\{ x\in\HH  \ \middle| G(x) \le G(x_0) \right\},
	\]
which is a bounded, convex set containing the minimizer. Then, assuming that $G$ is $L_B$--smooth on $B$ and $\mu$--strongly convex, and that the step size satisfies $s \in (0,2/(L_B+\mu)]$, it is possible to show that $x_k \in B$ for all $k \geq 0$. Moreover, in this setting, the scheme converges exponentially fast to the minimizer (see \citep{feng2017preconditioned, Chen2018Convergence, nesterov_2014}). In particular, if $s= 2/(L_B+\mu)$, then
	\begin{equation}
	\label{eq:convGD}
\norm{x_k - x^*}{\LL} \leq \left( \frac{1-\rho}{1+\rho} \right)^k \norm{x_0-x^*}{\LL}.
	\end{equation}

	\begin{algorithm}
 
  \caption{Preconditioned accelerated gradient descent method (PAGD)}
 \label{alg:PAGD}
  \KwData{$G$: The objective}
  \KwData{$\eta>0$: The friction coefficient}
  \KwData{$s>0$: The step size}
  \KwData{$x_0 \in \HH$: The initial guess}
  
  \KwResult{The sequence $\{x_k\}_{k\geq1}$ that approximates $x^*$, the minimizer of $G$}

  \KwSty{Define: $\theta=\eta\sqrt{s}$ and  $\lambda=\frac{1-\theta}{1+\theta}$}\;
  \KwSty{Set:} $x_{-1}=v_0 = x_0 \in\HH$\;
  \For{$k\geq0$}{
    \begin{align}
    	y_{k} &=x_{k}+\lambda(x_{k}-x_{k-1}),
    	\label{PAGD-y-upd}
    	\\
      x_{k+1} &=y_k- s \LLinv G'(y_k) ,
    \label{PAGD-x-upd}
    \\ 
      v_{k+1}&=x_k+\oneover{\theta}(x_{k+1}-x_k). 
    \label{PAGD-v-upd}
    \end{align}
  }
	\end{algorithm}

To improve on the convergence of GD (Algorithm~\ref{alg:PGD} with  $\LL=\mathfrak{R}_{\HH}^{-1}$), \citet{nesterov1983} devised an algorithm, which ``accelerates'' the rate of convergence of GD. The improved algorithm is commonly known as \emph{Nesterov's accelerated gradient descent method} (AGD). The preconditioned version of this scheme, PAGD, is presented in Algorithm \ref{alg:PAGD}. Roughly speaking, it computes an extrapolation, \eqref{PAGD-y-upd}, takes a gradient step there, \eqref{PAGD-x-upd}, and repeats the same process. Notice that an actual implementation does not need to compute the sequence $\{v_k\}_{k\geq0}$. We need it for the theoretical analysis.  As we will see in Section~\ref{sec:PAGD-conv}, for convergence, the algorithm must satisfy the condition $s\le 1/L_B$ and $\eta\le \sqrt \mu$ where $L_B>0$ is the (local) Lipschitz smoothness constant of $G$ with respect to a bounded convex neighborhood of the minimizer, $B$, and $\mu$ is the strong convexity constant.

It must be noted that PAGD, as presented in Algorithm~\ref{alg:PAGD}, is practical only if the objective functional is $\mu$--strongly convex ($\mu>0$). Otherwise, a convergence result may not be available. There exists a more general scheme, which one may call \emph{accelerated gradient descent method with variable weights} (see \citep[p. 78]{nesterov_2014}), that is applicable to merely convex objectives. We do not discuss this case here.

Let us now compare the performances of GD and AGD (Algorithm \ref{alg:PGD} and Algorithm \ref{alg:PAGD} with $\LL=\mathfrak{R}^{-1}_{\HH}$ respectively) by comparing $G(x_k)-G^*$, where $k$ is the number of iterations and $G^*=G(x^*)$ is the minimum of $G$.
To the best of our knowledge, the existing results on AGD are established under the assumption that the objective is \emph{globally} Lipschitz smooth. Thus, for the rest of the summary of this section, the objective $G$ is assumed to be (globally) $L$--smooth.
If GD is applied to a (merely) convex, $L$--smooth objective functional with a step size condition $0<s\leq 1/L$, then we have a first order convergence in the objective functional, i.e., $G(x_k)-G^*\le\bigo{1/k}$ as $k\goesto\infty$ (see \citep[Corollary 2.1.2]{nesterov_2014}). On the other hand, AGD with variable weights (the more general version mentioned above) provides a second order convergence, that is, $G(x_k)-G^*\le\bigo{1/k^2}$ as $k\goesto\infty$. If the objective is, in addition, $\mu$--strongly convex, the convergence rates of the two schemes become exponential. Specifically, estimate \eqref{eq:convGD} and the quadratic traps show that the convergence rate of GD is $G(x_k)-G^*\le\bigo{(\frac{1-\rho}{1+\rho})^{2k}}$ as $k\goesto\infty$, where we recall that $\rho=\mu/L$. This is in contrast to AGD, which converges with a rate of $G(x_k)-G^*\le\bigo{(1-\sqrt{\rho})^k}$ as $k\goesto\infty$; see \citep[Theorem 2.2.3]{nesterov_2014}. If $\rho\ll1$, this acceleration can be significant. 
As we will see later, PAGD achieves the same rate of exponential convergence even if the objective is locally Lipschitz smooth instead of the Lipschitz smoothness being imposed globally.

\section{An ODE model for PAGD}
\label{sec:ODE-model}

	We study a continuous time analogue of PAGD, a second order ODE, inspired by \citep{su}. As we will see, the discussion in this section turns out to be informative. It not only provides an intuitive understanding of Nesterov's acceleration, but also guides us to important results at the discrete level.  As mentioned in the introduction, there are recent works that arrive at the same conclusion for some parts of our results using similar ideas.  However, our unique contributions rely on our specific layout of various quantities and calculations. Thus, we include such details in a condensed manner while referring to existing work otherwise.

To streamline the discussion, we start by directly introducing the initial value problem (IVP) whose certain discretization leads to PAGD:
\begin{equation}
  \ddot{X}(t)+2\eta\dot{X}(t)+\LLinv G'(X(t))=0, \ t>0, \qquad 
  X(0)=x_0, \qquad \dot{X}(0)=0. 	
\label{the-IVP}
\end{equation}

Interestingly, this is the same system as what  \citet{polyak1964} had in mind when he proposed the \emph{heavy ball} method.
See Appendix \ref{app-PAGD-to-IVP} for its derivation. Conversely, PAGD can be viewed as a discretization of this IVP although not every choice in the process can be seen as natural or intuitive. It is
given in Appendix \ref{subsec:discretization}. Note that it involves some ingredients appearing in Section \ref{sub:RateOfConvergence}.

	\begin{rmk}[physical interpretation]
	\label{rmk:ODE-phys-interp}

	\begin{figure}
		{\includegraphics[height=0.4\textwidth]{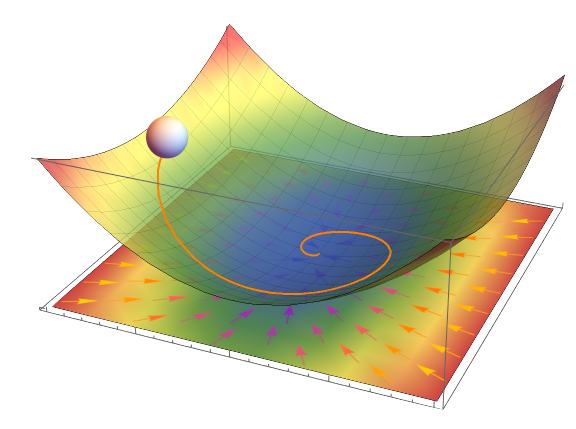}}
		\caption{A rolling ball system. The IVP \eqref{the-IVP} describes a ball of unit mass rolling down a bowl-shaped potential landscape with a constant friction coefficient.}
				\label{fig:rollingball}
	\end{figure}
	
This IVP \eqref{the-IVP} describes the motion of a ball of unit mass in the potential $G$ with friction coefficient $2\eta$ which starts from the initial position $x_0$ at rest; see Figure \ref{fig:rollingball}.
Our physical intuition suggests that the ball will converge to its minimal point as it exhausts the initial energy under the action of friction. If the friction, quantified by $\eta$, is too small it will oscillate much as it reaches the minimal point and will converge only after a long travel. On the other hand, if the friction is too large, it will not move sufficiently rapidly, and this, in turn, will also lead to a slow convergence. 

Let us compare this with another physical system. We can interpret the physics of the gradient flow as a limiting case of the same dynamics. The gradient flow 
$ %\[
\dot X(t)=-\LLinv G'(X(t))
$ %	\]
can be viewed, up to a constant factor $2\eta$, as a massless limit of the IVP \eqref{the-IVP}. That is, a physical thought experiment suggests that the surroundings hold the particle back as soon as it gets accelerated since it is so light. A real life example of this kind is a very viscous fluid, such as honey, flowing down a bowl. Our physical experience suggests that it will not oscillate and will flow along the steepest descent direction every moment. However, it will reach the bottom slower than the rolling ball will if the friction is appropriately strong.
	\ermk\end{rmk}

\subsection{Analysis of the IVP}
\label{sub:ExUniqueSmooth}

As one can expect from the fact that the IVP \eqref{the-IVP} describes a concrete physical situation, its solution possesses good properties. In this and the following section, however, we need a slightly stronger Lipschitz condition on $G'$ than in the discrete level discussion.

	\begin{lem}[existence and uniqueness]
	\label{lem:ExistUniqueTheIVP}
Suppose that $G:\HH\goesto\RR$ is $\mu$--strongly convex and locally Lipschitz smooth in the strong sense, i.e., \eqref{def:Lip-str} holds. Then, for any $T>0$, there exists a unique solution $X\in C^2(0,T;\HH)$ to the initial value problem \eqref{the-IVP} and the solution obeys the following energy identity
	\begin{equation}
\half\norm{\dot X(t)}{\LL}^2 +G(X(t))-G^*=G(x_0)-G^*-2\eta \int_0^t\norm{\dot{X}(\tau)}{\LL}^2 \diff \tau\quad \forall t\geq 0.
 \label{eql:naive-egy}
 \end{equation}
 Consequently, the solution exists for all $t\in [0,\infty)$ and it is twice continuously differentiable.
 	\end{lem}
	\begin{proof}
See \citep[Theorem 3.1 and Proposition 4.2]{Attouch2000} for the existence, uniqueness, and smoothness. For the energy law, take the $\LL$--inner product of the first equation of \eqref{the-IVP}  with $\dot X$ and integrate over time $\tau \in[0,t]$. 
\end{proof}

	\begin{rmk}[smoothness of the solution]
The fact that
$%\[
  X \in C^2((0,\infty); \HH) \cap C([0,\infty); \HH)
$ %\]
justifies the manipulations we will carry out when we derive the IVP \eqref{the-IVP} in Appendix \ref{app-PAGD-to-IVP}.
	\ermk\end{rmk}

	\subsection{Convergence to equilibrium}
	\label{sub:RateOfConvergence}
We now wish to prove that the solution to the IVP \eqref{the-IVP} with $G$ being locally Lipschitz smooth and $\mu$--strongly convex converges to its attractive steady state solution as $t\goesto\infty$, which is the minimal point in this case, at a matching rate with that of PAGD. This is one of the highlights of this work.
To this end, we introduce an auxiliary variable
\[
V(t)=X(t)-x^*+\frac{1}{\eta}\dot{X}(t)
\]
so that the first equation of the IVP \eqref{the-IVP} can be rewritten
\begin{equation}
\eta \dot{V}(t) +\eta\dot{X}(t)+\LLinv G'(X(t))=0.
\label{theV-ODE}
\end{equation}
We also introduce an energy
	\begin{equation}
E(X,V)=\frac{\eta}{2} \norm{V}{\LL}^2+\frac{1}{\eta}(G(X)-G^*),
	\label{lyapunov}	
	\end{equation}
where we recall $G^*=G(x^*)=\min_{x\in\HH} G(x)$. We will show that $E$ is a \emph{Lyapunov} function for the IVP \eqref{the-IVP}. For notational convenience, set
$ %\begin{align}
E_0=E(x_0,x_0)=\oneover{\eta} (G(x_0)-G^*)+\frac{\eta}{2} \norm{x_{0}-x^*}{\LL}^2.
%\label{const:E_0}
$ %\end{align}

\begin{thm}[exponential decay]
\label{thm:Lyap-decay}
Let $G:\HH\goesto\RR$ be locally Lipschitz smooth in the strong sense and $\mu$--strongly convex. Denote by $X$ the unique solution to the IVP \eqref{the-IVP}.
If $\eta^2\le\mu$, the exponentially inflated energy $\mathcal{E}(t)=e^{\eta t} E(X(t),V(t))$ is nonincreasing. Consequently, the Lyapunov function \eqref{lyapunov} decays to zero at an exponential rate:
\begin{align}
  E(X(t),V(t))=\frac{\eta}{2} \norm{V(t)}{\LL}^2+\frac{1}{\eta}\left(G(X(t))-G^* \right)\le {e^{-\eta t}}E_0.
	\label{est:Lyap-decay}
\end{align}
\end{thm}
\begin{proof}
Existence and uniqueness of $X$ is guaranteed by Lemma~\ref{lem:ExistUniqueTheIVP}. Let us now prove the estimate \eqref{est:Lyap-decay}. Taking the inner product of \eqref{theV-ODE} with $V(t)$, and using the identity \eqref{3pt-id-}, we obtain, suppressing the time variable,  
\begin{equation}
  \begin{aligned}
    0&=\eta(V,\dot{V})_\LL +\eta (\dot X, X-x^*+\frac{1}{\eta}\dot X)_\LL +(\LLinv G'(X), X-x^*)_\LL \\ 
    &+\frac{1}{\eta}(\LLinv G'(X),\dot X)_\LL \\
    &=\eta (V,\dot{V})_\LL +\half \norm{\dot X}{\LL}^2 + \frac{\eta^2}{2}\norm{V}{\LL}^2 -\frac{\eta^2}{2}  \norm{X-x^*}{\LL}^2  +\pairing{G'(X)}{X-x^*} \\
    &+\frac{1}{\eta}\pairing{G'(X)}{\dot X}.
  \end{aligned}
\label{est1-ODE-P}
\end{equation}
The lower quadratic trap, \eqref{def:str-convex}, implies
\begin{equation}
  G(X)-G^* - \pairing{G'(X)}{X-x^*} \le -\frac{\mu}{2}\norm{X-x^*}{\LL}^2.
\label{est-str-con1-P}
\end{equation}
Substituting \eqref{est1-ODE-P} into the time derivative of the inflated energy and then using the above estimate \eqref{est-str-con1-P}, we have
\begin{equation}
  \begin{aligned}
  \dot{\mathcal{E}}(t)&= e^{\eta t}\left[ \frac{\eta^2}{2} \norm{V}{\LL}^2 +\eta(V,\dot V)_\LL + (G(X)-G^*)+\frac{1}{\eta}\pairing{G'(X)}{\dot X} \right] 	
  \\
  &=e^{\eta t}\left[ -\half \norm{\dot X}{\LL}^2  +\frac{\eta^2}{2}  \norm{X-x^*}{\LL}^2  -\pairing{G'(X)}{X-x^*}+ G(X)-G^* \right]	
  \\
  &\le -\half e^{\eta t} \left[ \norm{\dot X}{\LL}^2 +\frac{\mu-\eta^2}{2} \norm{X-x^*}{\LL}^2 \right].
  \end{aligned}
\label{der-engy-neg-P}	 
\end{equation}
The last term is always nonpositive provided $\eta^2\le\mu$, and this implies $\mathcal{E}(t)\le \mathcal{E}(0)=E_0$.  This completes the proof.	
\end{proof}

\begin{rmk}[physical interpretation]
We can rigorously explain the physical intuition given in Remark \ref{rmk:ODE-phys-interp} through Theorem \ref{thm:Lyap-decay} and its proof. If the friction, quantified by $\eta$, is too small the decay to the attraction point is slow as $\eta$ governs the decay rate $e^{-\eta t}$. On the other hand, If the friction is too large, say $\eta>\sqrt{\mu}$, then we cannot guarantee the boundedness of $\mathcal{E}(t)$.
\ermk\end{rmk}

\section{An energy approach to convergence of PAGD}
\label{sec:PAGD-conv}

In this section, we prove the existence of an invariant set of PAGD and its exponential convergence in the objective as well as in the residual when it is applied to a strongly convex, locally Lipschitz smooth objective.  We follow the ODE arguments developed in Section~\ref{sec:ODE-model}.  Throughout this section, we assume $\eta=\sqrt \mu$, the optimal choice for the friction coefficient in view of Theorem \ref{thm:Lyap-decay}. Note that this does not undermine generality. If $\tilde\mu$ is the largest strong convexity constant of $G$, that is, the supremum of $\mu$'s that satisfies the strong convexity, \eqref{strconv-coerc-prec}, then any $\mu\in (0,\tilde \mu]$ can be taken as a (non-optimal) strong convexity constant. Thus, the general case $\eta^2 \le \tilde \mu$ corresponds to $\eta^2=\mu \le \tilde \mu$, the optimal friction coefficient associated with a non-optimal strong convexity constant. 

As a first step, we show that the assumption of the local Lipschitz smoothness is sufficient for our analysis, as the iterates lie within a bounded set. We first show that, for every $k\geq 0$, the $y_k$ iterate of PAGD lies in the segment between $x_k$ and $v_k$. This is used frequently in the convergence proof.

\begin{lem}[convex hull]\label{lem:xyv-rel}
For every $k \geq 0$, the iterates constructed in PAGD, described in Algorithm~\ref{alg:PAGD}, satisfy $y_k\in\overline{x_{k} v_{k}}$. Specifically, they satisfy the following four equivalent equations:
\begin{equation}
  \begin{dcases}
    y_k=\frac{1}{1+\theta} x_k +\frac{\theta}{1+\theta}v_k, & \qquad x_k=(1+\theta)y_k -\theta v_k, \\
    v_k= \left(1+\oneover{\theta} \right)y_k-\oneover{\theta}x_k, & \qquad x_k-y_k=\theta(y_k-v_k).
  \end{dcases}
\label{xyv-rel}  
\end{equation}
\end{lem}
\begin{proof}
If $k=0$ this is trivial since $x_0=y_0=v_0$. For $k\ge1$, we eliminate $x_{k-1}$ from \eqref{PAGD-y-upd} and \eqref{PAGD-v-upd} with the index being $k-1$ to get
\[
  \left(1-\oneover{\theta}\right)y_k+\lambda v_k=\left( (1+\lambda)(1-\oneover{\theta})+\frac{\lambda}{\theta}\right)x_k=-\frac{\lambda}{\theta}x_k.
\]
Rearranging terms and using $\lambda=\frac{1-\theta}{1+\theta}$, we obtain the equalities that are listed above.	
\end{proof}

We now show that there is an invariant set for the iterates of PAGD.

\begin{lem}[invariant set]
\label{lem:local-Lip}
Assume that the objective $G:\HH \goesto \RR$ is $\mu$--strongly convex and locally Lipschitz smooth. Define
	\begin{align}
\mathfrak{B}= \left\{x\in\HH \ \middle| \ \norm{x-x^*}{\LL}\le R \right\},
\label{def:invariant-set}
\end{align}
where $R=R_1+\frac{1}{\eta}R_2$, $R_1=\sqrt{\frac{2}{\mu}(G(x_0)-G^*)}$, $R_2=\sqrt{2r(G(x_0)-G^*)}$, and $r>1$.
Let PAGD, as described in Algorithm~\ref{alg:PAGD}, be implemented with a step size rule
\begin{equation}
\label{eq:sRule}
  s \in \left( 0, \min \left\{L_B^{-1}, \left(\frac{r-1}{r+1} \right)^2 \mu^{-1}\right\} \right],
\end{equation}
where $L_B$ the local Lipschitz smoothness constant of $G$ associated to the set $\mathfrak{B}$. Then, for all $k \geq 0$, we have that $\norm{x_{k}-x^*}{\LL}\le R_1$, hence $x_k\in \mathfrak{B}$, and $y_k, v_k\in \mathfrak{B}$.
\end{lem}
\begin{proof}
The outline of this proof is simple although it is long. We mimic the energy law developed in Section~\ref{sub:ExUniqueSmooth} to obtain a bound on the distance between the main iterates and the minimizer and that on the speed. Once we get the bounds, it is easy to prescribe an appropriate ball, which will be our invariant set.

We will prove the statement by induction.  For $k=0$, the statement is trivial since $x_0=y_0=v_0$ and the strong convexity implies $\norm{x_0-x^*}{\LL}\le R_1$. Suppose that $\norm{x_k-x^*}{\LL}\le R_1$ (hence $x_k\in \mathfrak{B}$) and $y_k, v_k\in \mathfrak{B}$ are true for $k=0,1,2,\cdots,N$. We need to show that $\norm{x_{N+1}-x^*}{\LL}\le R_1$ (hence $x_{N+1}\in \mathfrak{B})$ and $v_{N+1}\in \mathfrak{B}$, then Lemma \ref{lem:xyv-rel} implies $y_{N+1}\in\overline{x_{N+1}v_{N+1}}\subset \mathfrak{B}$ since $\mathfrak{B}$ is a convex set as a sublevel set of a convex function.

Note that the condition $s\le (\frac{r-1}{r+1})^2 \mu^{-1}$, which is implied by \eqref{eq:sRule}, ensures $\lambda^{-1}$ to be bounded above since 
	\begin{equation}
\frac{1}{\lambda}=\frac{1+\sqrt{s\mu}}{1-\sqrt{s\mu}}\le r .
	\label{est10}
	\end{equation}

First, a similar argument to \citep[Proposition 4.6]{Chen2018Convergence} shows that the $x_{N+1}$ update from $y_N$ is a descent step in terms of $G$. 
That is, the section of $G$ across the line $\overleftrightarrow{y_N x_{N+1}}$ also inherits the strong convexity and the local Lipschitz smoothness with the same constants on the one-dimensional affine subset
\[
  B_{N+1} = \left\{x=y_N-\tau \LLinv G'(y_N)\in \HH \middle| \tau\in\RR \right\}.
\]
Let $S(\tau)=G(y_N-\tau \LLinv G'(y_N))$ denote the section.

Since we know that $y_N\in \mathfrak{B}$, we can bound $S$ in a neighborhood of $\tau = 0$ using the upper quadratic trap 	
	\begin{align*} U(\tau)&:=G(y_N)+\pairing{G'(y_N)}{-\tau\LLinv G'(y_N)}+\frac{L_B}{2}\norm{-\tau\LLinv G'(y_N)}{\LL}^2
	\\
	&=G(y_N)-\tau\norm{ G'(y_N)}{\LLinv}^2+\frac{L_B\tau^2}{2}\norm{ G'(y_N)}{\LLinv}^2 .
	\end{align*}

Observe that $S(0)=U(0)=G(y_N)$, that $U(\tau)$ is decreasing around $\tau=0$ since ${\diff U}/{\diff \tau}(0)=-\norm{ G'(y_N)}{\LLinv}^2\le0$, and that the optimal step size to minimize $U$ is $1/L_B$ since ${\diff U}/{\diff \tau} (1/L_B)=0$. This implies that $S(s)\le U(s)\le U(0)$ for any $s\in[0,2/L_B]$. Moreover, for $s\in[0,1/L_B]$, we have
\begin{equation}
  \begin{aligned}
G(x_{N+1})&=S(s)\le U(s)=G(y_N)-s\norm{ G'(y_N)}{\LLinv}^2+\frac{L_B s^2}{2}\norm{ G'(y_N)}{\LLinv}^2
	\\
	&\le G(y_N)-\frac{s}{2}\norm{ G'(y_N)}{\LLinv}^2,    
  \end{aligned}
\label{est:descent}
\end{equation}
which is the desired descent property in $G$ from $y_N$ to $x_{N+1}$.

Now, we want to mimic the energy argument that we carried out in Section~\ref{sub:ExUniqueSmooth}. Substitute \eqref{PAGD-y-upd} into \eqref{PAGD-x-upd}, and add and subtract $x_k-x_{k-1}$, to obtain the discrete counterpart of \eqref{the-IVP}
\begin{equation}
x_{k+1}-2x_k+x_{k-1}+(1-\lambda)(x_k - x_{k-1})+s\LLinv G'(y_k)=0.
\label{eql1} 
\end{equation} 
Note that defining $x_{-1}:=x_0$ allows us to extend this equality to the case $k=0$.
Take the $\LL$--inner product of this identity with $x_k-x_{k-1}$ and add for $0\le k \le N$. Then, using \eqref{3pt-id+}, the first term telescopes to simplify
	\begin{align*}
&\quad \sum_{k=0}^{N}(x_{k+1}-2x_k+x_{k-1},x_k-x_{k-1} )_\LL
\\
&=\half \sum_{k=0}^{N}\left( \norm{x_{k+1}-x_k}{\LL}^2-\norm{x_{k+1}-2x_k+x_{k-1}}{\LL}^2-\norm{x_k-x_{k-1}}{\LL}^2 \right) 
	\\
&=\half  \norm{x_{N+1}-x_N}{\LL}^2 -\half \sum_{k=0}^{N}\norm{x_{k+1}-2x_k+x_{k-1}}{\LL}^2.
	\end{align*}
We leave the second term as it is. For the third term, using \eqref{PAGD-y-upd}, $G(y_k)-G(x_k) \le \pairing{G'(y_k)}{y_k-x_k} $ from convexity, and \eqref{est:descent}, it follows
\begin{equation}
  \begin{aligned}
&\quad s\sum_{k=0}^{N} \pairing{G'(y_k)}{x_k-x_{k-1}}= \frac{s}{\lambda}\sum_{k=0}^{N} \pairing{G'(y_k)}{y_k-x_{k}}
\\
&\ge\frac{s}{\lambda}\sum_{k=0}^{N}\left( G(y_k)-G(x_k)\right)
\ge \frac{s}{\lambda}\sum_{k=0}^{N}\left( G(x_{k+1})-G(x_k)+\frac{s}{2}\norm{G'(y_k)}{\LLinv}^2 \right)
\\
&=\frac{s}{\lambda}G(x_{N+1})-\frac{s}{\lambda}G(x_0) +\frac{s^2}{2\lambda}\sum_{k=0}^{N}\norm{G'(y_k)}{\LLinv}^2 .    
  \end{aligned}
\label{est5}
\end{equation}

Gathering all the three terms together and rearranging, we get
\begin{equation}
  \begin{aligned}
\half  \norm{x_{N+1}-x_N}{\LL}^2 +\frac{s}{\lambda}G(x_{N+1})
\le  
\frac{s}{\lambda}G(x_0)+\half\sum_{k=0}^{N}\norm{x_{k+1}-2x_k+x_{k-1}}{\LL}^2&  
\\
\quad  -\frac{s^2}{2\lambda}\sum_{k=0}^{N}\norm{G'(y_k)}{\LLinv}^2 -(1-\lambda)\sum_{k=0}^{N}\norm{x_k-x_{k-1}}{\LL}^2.&    
  \end{aligned}
\label{est8}
\end{equation}

Similarly, take the $\LL$--inner product of \eqref{eql1} with $x_{k+1}-x_k$ and sum over $0\le k \le N$. This time, use \eqref{3pt-id-} for the first term to get
	\begin{align*}
&\quad \sum_{k=0}^{N}(x_{k+1}-2x_k+x_{k-1},x_{k+1}-x_{k} )_\LL
\\
&=\half  \norm{x_{N+1}-x_N}{\LL}^2 +\half \sum_{k=0}^{N}\norm{x_{k+1}-2x_k+x_{k-1}}{\LL}^2.
\end{align*}
For the second term, using Cauchy-Schwarz and Young's inequality, we have
	\begin{align*}
&\quad(1-\lambda)\sum_{k=0}^{N}(x_k-x_{k-1},x_{k+1}-x_k)_\LL 
\\
&\ge -\frac{1-\lambda}{2}\sum_{k=0}^{N}\left( \norm{x_k-x_{k-1}}{\LL}^2+ \norm{x_{k+1}-x_{k}}{\LL}^2\right).
	\end{align*}
For the third term, use \eqref{PAGD-x-upd} and argue as in \eqref{est5} to get
	\begin{align*}
&\quad s\sum_{k=0}^{N} \pairing{G'(y_k)}{x_{k+1}-x_{k}}= s\sum_{k=0}^{N} \pairing{G'(y_k)}{y_k-s\LLinv G'(y_k)-x_{k}}
\\
&\ge -s^2\sum_{k=0}^{N} \norm{G'(y_k)}{\LLinv}^2 +s\sum_{k=0}^{N} \pairing{G'(y_k)}{y_k-x_{k}}
\\
&\ge sG(x_{N+1})-sG(x_0) -\frac{s^2}{2}\sum_{k=0}^{N}\norm{G'(y_k)}{\LLinv}^2.
\end{align*}
Gathering all these estimates we get
\begin{equation}
	\begin{aligned}
&\half  \norm{x_{N+1}-x_N}{\LL}^2+sG(x_{N+1})
\le sG(x_0) -\half \sum_{k=0}^{N}\norm{x_{k+1}-2x_k+x_{k-1}}{\LL}
\\
 &\quad +\frac{s^2}{2}\sum_{k=0}^{N}\norm{G'(y_k)}{\LLinv}^2
+\frac{1-\lambda}{2}\sum_{k=0}^{N}\left( \norm{x_k-x_{k-1}}{\LL}^2+ \norm{x_{k+1}-x_{k}}{\LL}^2\right).
\end{aligned}
\label{est7}
	\end{equation}
Add \eqref{est8} and \eqref{est7} and rearrange, then after some cancellations, it follows
	\begin{align}
&\quad \frac{1+\lambda}{2}\norm{x_{N+1}-x_N}{\LL}^2+s\left(1+\frac{1}{\lambda} \right)(G(x_{N+1})-G^*)
\nonumber
\\
&\le s\left(1+\frac{1}{\lambda}\right)(G(x_0)-G^*) 
-\frac{s^2}{2}\left(\frac{1}{\lambda}-1 \right)\sum_{k=0}^{N}\norm{G'(y_k)}{\LLinv}^2
\label{est:conv-residual}
\\
&\le s\left(1+\frac{1}{\lambda}\right)(G(x_0)-G^*),
	\end{align}
since $0<\lambda<1$. By removing the kinetic term from this estimate, strong convexity leads to
	\[
G(x_{0})-G^*\ge G(x_{N+1})-G^*\ge\frac{\mu}{2} \norm{x_{N+1}-x^*}{\LL}^2,
	\]
which implies
	\begin{equation}
\norm{x_{N+1}-x^*}{\LL}\le R_1,
\label{est9}
	\end{equation}
which, in turn, proves $x_{N+1}\in \mathfrak{B}$. Similarly, discarding the potential term, dividing through $\frac{s(1+\lambda)}{2}$, and using \eqref{est10}, we obtain
	\[
\norm{\frac{x_{N+1}-x_N}{\sz}}{\LL}\le \sqrt{\frac{2}{\lambda}(G(x_0)-G^*)}\le R_2 .
	\]
Then, from the definition of $v_{N+1}$ \eqref{PAGD-v-upd},
	\[
\norm{v_{N+1}-x^*}{\LL}\le \norm{x_{N}-x^*}{\LL}+\frac{1}{\eta}\norm{\frac{x_{N+1}-x_N}{\sz}}{\LL} \le R_1+\frac{1}{\eta}R_2=R,
\]
which implies $v_{N+1}\in \mathfrak{B}$. This completes the proof.
	\end{proof}
	
	\begin{rmk}[step size restriction]
		\label{rmk:sz-restriction}
The additional condition $s\le (\frac{r-1}{r+1})^2 \mu^{-1}$ on the step size is not restrictive at all in practice. For example, if $r=3$, we require that $s\le 1/4\mu$. The purpose of this condition is to bound $\lambda^{-1}$ as explained in the proof. However, $\lambda^{-1}$ becomes unbounded when $s\mu$ is close to $1$. If we set $s=1/L_B$, $s\mu$ is the (inverse) condition number and the (inverse) condition number being close to $1$ makes the problem more amenable because it means that $G$ is almost quadratic. Moreover, even from a theoretical point of view, as $r$ increases, the invariant set $\mathfrak{B}$ gets larger, which means $L_B^{-1}$ gets smaller, while $(\frac{r-1}{r+1})^2 \mu^{-1}$ approaches $\mu^{-1}$. Since $L_B^{-1}<\mu^{-1}$ (unless $G$ is perfectly quadratic), the second argument of the minimum in \eqref{eq:sRule} eventually becomes of no effect.
	\ermk\end{rmk}

\begin{cor}[convergence of residuals]
Assume that $G:\HH\goesto\RR$ is $\mu$--strongly convex and locally Lipschitz smooth. Suppose that PAGD, as described in Algorithm \ref{alg:PAGD}, is implemented with a step size that obeys condition \eqref{eq:sRule}, where $r>1$, $\mathfrak{B}$ is the invariant set given by \eqref{def:invariant-set}, and $L_B$ is the Lipschitz smoothness constant associated with $\mathfrak{B}$. In this setting, the residuals $\{G'(y_k)\}_{k\geq0}$ converge to zero in the $\LLinv$--norm at least $\ell^2$--fast. In other words,
	\begin{equation*}
	\sum_{k=0}^{\infty}\norm{G'(y_k)}{\LLinv}^2
	<\infty.
	\end{equation*} 
\end{cor}
\begin{proof}
Moving the summation term of \eqref{est:conv-residual} to the left hand side and dropping the other nonnegative terms, we have
	\begin{equation*}
\sum_{k=0}^{N}\norm{G'(y_k)}{\LLinv}^2
\le \frac{2(1+\lambda)}{s(1-\lambda)}(G(x_0)-G^*).
	\end{equation*}
Letting $N\goesto\infty$ completes the proof. 
\end{proof}

Of course, this result is far from optimal. An exponential convergence of the residuals in the $\LLinv$--norm will be proved in Corollary \ref{cor:PAGD-conv}.

We can now begin the proof of convergence \emph{per se}. 
We begin with an estimate for the discrete time derivative of the potential energy.

\begin{lem}[discrete derivative of potential energy]
\label{lem:disc-dt-G}
Let the objective $G:\HH\goesto\RR$ be $\mu$--strongly convex and locally Lipschitz smooth. Suppose that the step size in Algorithm~\ref{alg:PAGD} satisfies \eqref{eq:sRule}, where $r>1$, $\mathfrak{B}$ is the invariant set given by \eqref{def:invariant-set}, and $L_B$ is the Lipschitz smoothness constant associated with $\mathfrak{B}$. Then, we have that 	
\[
\oneover{\eta} \frac{G(x_{k+1})-G(x_k)}{\sz}\le -\frac{\sz}{2\eta} \norm{G'(y_k)}{\LLinv}^2 +\oneover{\theta}\pairing{G'(y_k)}{y_k-x_k}-\frac{\eta}{2\sz}\norm{x_k-y_k}{\LL}^2.	
\]
\end{lem}	
\begin{proof}
	Since we have an invariant set, $\mathfrak{B}$, we can utilize the local Lipschitz smoothness with respect to it. Combined with $sL_B\le1$, it leads to
\begin{align}
  G(x_{k+1})=G(y_k-s\LLinv G'(y_k))
  \le G(y_k)-\frac{s}{2}	\norm{G'(y_k)}{\LLinv}^2.
  \label{est:desc-step}
\end{align}
Using this estimate, the strong convexity of $G$  yields 
\begin{align*}
  G(x_{k})&\ge G(y_k) +\pairing{G'(y_k)}{x_k-y_k}+\frac{\mu}{2}\norm{x_k-y_k}{\LL}^2
  \\
  &\ge G(x_{k+1})+\frac{s}{2}	\norm{G'(y_k)}{\LLinv}^2	+\pairing{G'(y_k)}{x_k-y_k}+\frac{\mu}{2}\norm{x_k-y_k}{\LL}^2.
\end{align*}
Rearranging the last estimate, multiplying through by $1/\theta$, and recalling $\theta =\sqrt{s\mu}, \eta=\sqrt \mu$, we obtain the desired result.	
\end{proof}

We also need an analogue of \eqref{est1-ODE-P}, a certain relation derived from the scheme.

\begin{lem}[discrete analogue of \eqref{est1-ODE-P}]
\label{lem:disc-newODE}
The iterates constructed by PAGD, as described in Algorithm~\ref{alg:PAGD}, satisfy
\begin{multline}
\frac{\eta}{2\sz}\left( \norm{v_{k+1}-x^*}{\LL}^2- \norm{v_{k}-x^*}{\LL}^2\right)+\oneover{\theta}\pairing{G'(y_k)}{y_k-x_k}+\frac{\eta^2}{2}\norm{v_k-x^*}{\LL}^2 \\ 
-\frac{\eta}{2\sz}\norm{v_{k+1}-v_k}{\LL}^2 +\oneover{2s}\norm{y_k-x_k}{\LL}^2-\frac{\eta^2}{2}\norm{y_k-x^*}{\LL}^2+\pairing{G'(y_k)}{y_k-x^*}=0.
\label{est-scheme}
\end{multline}		
\end{lem}
\begin{proof}
Substituting \eqref{PAGD-x-upd} in \eqref{PAGD-v-upd}, and using the relations \eqref{xyv-rel}, we have
\[
v_{k+1}=x_k+\oneover{\theta}(y_k-x_k)-\frac{\sz}{\eta}\LLinv G'(y_k)
=x_k+v_k-y_k-\frac{\sz}{\eta}\LLinv G'(y_k).
\]
Rearranging, and multiplying through by $\frac{\eta}{\sz}$, we obtain the discrete analogue of the ODE \eqref{theV-ODE}
\begin{equation}
  \eta \frac{v_{k+1}-v_k}{\sz}+\eta\frac{y_k-x_k}{\sz}+\LLinv G'(y_k)=0.
\label{1line-scheme}	
\end{equation}
The discrete analogue of $V(t)$ is $v_k-x^*$, so following the proof of Theorem~\ref{thm:Lyap-decay}, we now take the $\LL$--inner product of \eqref{1line-scheme} with $v_k-x^*$ to obtain
\begin{equation}
  \frac{\eta}{\sz}\iprd{v_{k+1}-v_k}{v_k-x^*}_\LL+\frac{\eta}{\sz}\iprd{y_{k}-x_k}{v_k-x^*}_\LL+\pairing{G'(y_k)}{v_k-x^*}=0.
\label{est3}
\end{equation}
Using \eqref{3pt-id+}, the first term can be rewritten as
\begin{multline*}
  \frac{\eta}{\sz}\iprd{v_{k+1}-v_k}{v_k-x^*}_\LL= \\
  \frac{\eta}{2\sz}\left( \norm{v_{k+1}-x^*}{\LL}^2- \norm{v_{k}-x^*}{\LL}^2\right)-\frac{\eta}{2\sz}\norm{v_{k+1}-v_k}{\LL}^2. 
\end{multline*}
For the second term in \eqref{est3}, we use relations \eqref{xyv-rel}, and then the identity \eqref{3pt-id-} to get
\begin{align*}
\frac{\eta}{\sz}\iprd{y_{k}-x_k}{v_k-x^*}_\LL&=\frac{\eta}{\sz}\oneover{\theta}\iprd{y_{k}-x_k}{\theta v_k-\theta x^*}_\LL
\\
&=\oneover{s}\iprd{y_{k}-x_k}{y_k-x_k+\theta(y_k-x^*)}_\LL
\\
&=\oneover{2s}\norm{y_k-x_k}{\LL}^2+\frac{\eta^2}{2}\norm{v_k-x^*}{\LL}^2   
-\frac{\eta^2}{2}\norm{y_k-x^*}{\LL}^2.
\end{align*}
Finally, for the third term of \eqref{est3}, we use \eqref{xyv-rel} similarly to the above, then it follows
\begin{align*}
\pairing{G'(y_k)}{v_k-x^*}&=\oneover{\theta}	\pairing{G'(y_k)}{\theta v_k-\theta x^*}
\\
&=\oneover{\theta}	\pairing{G'(y_k)}{y_k-x_k}+	\pairing{G'(y_k)}{y_k-x^*}.
\end{align*}
Then, the desired result follows upon combining the last three identities.		
\end{proof}

We need one more relation between the iterates.
 
\begin{lem}[relation between iterates]
\label{lem:vk1-vk}
The iterates constructed by PAGD, as described in Algorithm~\ref{alg:PAGD}, satisfy
\[
  \frac{\eta}{2\sz}\norm{v_{k+1}-v_k}{\LL}^2=	\frac{\eta}{2\sz}\norm{x_k-y_k}{\LL}^2+\frac{\sz}{2\eta}\norm{G'(y_k)}{\LLinv}^2+\pairing{G'(y_k)}{y_k-x_k}.
\]
\end{lem}
\begin{proof}
Combine \eqref{PAGD-v-upd} and  the relations \eqref{xyv-rel}, and then use \eqref{PAGD-x-upd} to obtain
\begin{align*}
  v_{k+1}-v_k&=	x_k+\frac{1}{\theta}(x_{k+1}-x_k)+\oneover{\theta}x_k-\left(1+\oneover{\theta} \right)y_k
  \\
  &=x_k-y_k+\oneover{\theta}(x_{k+1}-y_k)=x_k-y_k-\frac{\sz}{\eta}\LLinv G'(y_k)	.
\end{align*}	
Take $\LL$--norm square on both sides and then multiply by $\frac{\eta}{2\sz}$. 	
\end{proof}

We are now in a position to prove the main result of this section, the exponential convergence of PAGD using energy arguments. The following result and its consequences are another of the main contributions of this work. To state it, we recall that the Lyapunov function of \eqref{the-IVP} $E:\HH^2\goesto\RR$ is defined in \eqref{lyapunov} and that $E_0$ is its value at the initial state as mentioned in Theorem \ref{thm:Lyap-decay}. %\eqref{const:E_0}.

\begin{thm}[exponential decay]
\label{thm:exp-conv}
Let the objective $G:\HH\goesto\RR$ be locally Lipschitz smooth and $\mu$--strongly convex. If PAGD, as described in Algorithm~\ref{alg:PAGD}, is applied to approximate $x^*=\argmin_{x\in\HH}G(x)$ with a step size satisfying \eqref{eq:sRule}, where $r>1$, $\mathfrak{B}$ is the invariant set given by \eqref{def:invariant-set}, and $L_B$ is the Lipschitz smoothness constant associated with $\mathfrak{B}$, then the Lyapunov function \eqref{lyapunov} decays exponentially along the iterates $\{x_k\}_{k\ge0}$. More specifically, for $k\geq0$, we have
\begin{equation}
  E(x_{k+1},v_{k+1}-x^*) \le (1-\theta)	E(x_{k},v_{k}-x^*),
  \quad
  E(x_{k},v_{k}-x^*) \le (1-\theta)^k E_0.
\label{PAGDC-total-dec}
\end{equation}	
\end{thm}
\begin{proof}
Define, for $k \geq 0$, $\mathcal{E}_k=(1-\theta)^{-k} E(x_k, v_k-x^*)$, which is the discrete analogue of the exponentially inflated energy in the the proof of Theorem~\ref{thm:Lyap-decay}. To simplify notation, we set $C_{\theta,k}=(1-\theta)^{-(k+1)}>0$. Then, similarly to the ODE case, one can show the discrete time derivative of $\mathcal{E}_k$ is nonpositive as follows. First, we simply use the forward difference time derivative, rearrange, and use Lemma \ref{lem:disc-dt-G} to get
\begin{align*}
\frac{\mathcal{E}_{k+1}-\mathcal{E}_k}{\sz} &=
\oneover{\sz}\left[  (1-\theta)^{-(k+1)}\left(\oneover{\eta} (G(x_{k+1}) -G^*) +\frac{\eta}{2}\norm{v_{k+1}-x^*}{\LL}^2\right)  \right. \\
&- \left. (1-\theta)^{-k}  \left(\oneover{\eta}(G(x_{k})-G^*)+\frac{\eta}{2}\norm{v_{k}-x^*}{\LL}^2\right) \right] \\
&= C_{\theta,k} \left[ \oneover{\eta}\frac{G(x_{k+1})-G(x_k)}{\sz}+(G(x_k)-G^*) \right. \\
&+\left. \frac{\eta}{2\sz}(\norm{v_{k+1}-x^*}{\LL}^2-\norm{v_{k}-x^*}{\LL}^2)+\frac{\eta^2}{2}\norm{v_k-x^*}{\LL}^2 \right] \\
&\le C_{\theta,k} \left[-\frac{\sz}{2\eta} \norm{G'(y_k)}{\LLinv}^2 -\frac{\eta}{2\sz}\norm{x_k-y_k}{\LL}^2+(G(x_k)-G^*) \right.\\
& +\oneover{\theta}\pairing{G'(y_k)}{y_k-x_k}+\frac{\eta}{2\sz}(\norm{v_{k+1}-x^*}{\LL}^2-\norm{v_{k}-x^*}{\LL}^2) \\
&+ \left. \frac{\eta^2}{2}\norm{v_k-x^*}{\LL}^2 \right].
\end{align*}
We continue by using Lemma \ref{lem:disc-newODE} and then Lemma \ref{lem:vk1-vk}, then it follows
\begin{align*}
\frac{\mathcal{E}_{k+1}-\mathcal{E}_k}{\sz} &\le C_{\theta,k} \left[-\frac{\sz}{2\eta} \norm{G'(y_k)}{\LLinv}^2 -(\frac{\eta}{2\sz}+\frac{1}{2s})\norm{x_k-y_k}{\LL}^2+(G(x_k)-G^*) \right.\\
&\left.+\frac{\eta}{2\sz}\norm{v_{k+1}-v_k}{\LL}^2 +\frac{\eta^2}{2}\norm{y_k-x^*}{\LL}^2+\pairing{G'(y_k)}{x^*-y_k} \right]
\\
&= C_{\theta,k} \left[- \oneover{2s}\norm{y_k-x_k}{\LL}^2+(G(x_k)-G^*) 	+\pairing{G'(y_k)}{y_k-x_k} \right. \\
& \left. +\pairing{G'(y_k)}{x^*-y_k} +\frac{\eta^2}{2}\norm{y_k-x^*}{\LL}^2 \right].
\end{align*}
Finally, add and subtract $G(y_k)$ from the last expression, and use the following estimates, which are simple rearrangements of the lower and upper quadratic traps,
	\begin{align*}
	G(y_k)-G^*+\pairing{G'(y_k)}{x^*-y_k}\le-\frac{\mu}{2}\norm{y_k-x^*}{\LL}^2
	\\
	G(x_k)-G(y_k)+\pairing{G'(y_k)}{y_k-x_k}\le\frac{L_B}{2}\norm{x_k-y_k}{\LL}^2,
	\end{align*}
then we arrive at
\begin{align*}
\frac{\mathcal{E}_{k+1}-\mathcal{E}_k}{\sz} &\le C_{\theta,k} \left[G(x_k)-G(y_k) 	+\pairing{G'(y_k)}{y_k-x_k}+G(y_k)-G^* \right. \\
&+ \left. \pairing{G'(y_k)}{x^*-y_k} - \oneover{2s}\norm{y_k-x_k}{\LL}^2+\frac{\eta^2}{2}\norm{y_k-x^*}{\LL}^2 \right]	\\
&\le C_{\theta,k} \left[ \half\left(L_B-\oneover{s}\right)\norm{y_k-x_k}{\LL}^2 \right].
\end{align*}   
The step size condition forces the last term to be nonpositive. Therefore, we conclude that $\{\mathcal{E}_{k}\}_{k\ge0}$ is nonincreasing, from which we obtain \eqref{PAGDC-total-dec}.
\end{proof}

The following estimates are evident.

\begin{cor}[rate of convergence] \label{cor:PAGD-conv}
In the setting of Theorem~\ref{thm:exp-conv}, we have that the iterates of PAGD, as described in Algorithm~\ref{alg:PAGD}, converge to $x^*$, the minimizer of $G$, at an exponential rate. More specifically, for a suitable $r>1$ the step size can be set $s=1/L_B$ and, in this case, for $k\geq 0$,
\begin{equation}
  \oneover{\eta}(G(x_{k})-G^*) + \frac{\eta}{2} \norm{v_{k}-x^*}{\LL}^2 \le \left(1-\sqrt{\rho}\right)^{k} {E}_0,
\label{PAGDC-G-err-decay}	
\end{equation}
which implies
\begin{align}
  G(x_k)-G^*\le \left(1-\sqrt{\rho}\right)^{k} {\eta}	{E_0},
  \qquad
  \norm{x_k-x^*}{\LL} \le \left(1-\sqrt{\rho}\right)^\frac{k}{2}\sqrt{\frac{2{E}_0}{\eta}}.
\label{PAGDC-err-decay}	
\end{align}
Furthermore, we have exponential convergence in the $\LLinv$--norm of the residuals: for $k\geq 0$, 
	\begin{align}
\norm{G'(y_k)}{\LLinv}\le 3L_B \sqrt{\frac{2{E}_0}{\eta}}\left(1-\sqrt{\rho}\right)^\frac{k-1}{2}.
	\label{PAGD-G'-decay}
\end{align}
\end{cor}
\begin{proof}
We can choose an appropriate $r>1$ so that the step size condition \eqref{eq:sRule} reduces to $s\in(0,L_B^{-1}]$; see Remark \ref{rmk:sz-restriction}. Estimate \eqref{PAGDC-G-err-decay} and the first estimate of \eqref{PAGDC-err-decay} follow from \eqref{PAGDC-total-dec} upon setting $s = 1/L_B$. The second estimate of \eqref{PAGDC-err-decay} follows by applying strong convexity of $G$ to the first estimate of \eqref{PAGDC-err-decay}.

Next, from the estimate \eqref{est:desc-step}, one obtains 
	\begin{equation}
G^*\le G(x_{k+1})\le G(y_k)-\frac{1}{2L_B}\norm{G'(y_k)}{\LLinv}^2,
	\end{equation}
from which, one obtains the following by rearranging and then using the upper quadratic trap
	\begin{align}
	\norm{G'(y_k)}{\LLinv}\le\sqrt{2L_B (G(y_k)-G^*)}\le L_B\norm{y_k-x^*}{\LL}.
	\nonumber 
	\end{align}
In addition, we also have, from the definition of $y_k$ and $0<\lambda<1$,
\begin{align*}
\norm{y_k-x^*}{\LL}=\norm{x_k-x^*+\lambda(x_k-x_{k-1}\pm x^*)}{\LL}\le 2\norm{x_k-x^*}{\LL}+\norm{x_{k-1}-x^*}{\LL}.
\end{align*}
Combining the last two estimates and using \eqref{PAGDC-err-decay}, we obtain \eqref{PAGD-G'-decay}.
\end{proof}

	\begin{rmk}[total energy]
	\label{rmk:each-egy-osc}
The exponential decrease of the ``total energy'' at every step, given in \eqref{PAGDC-G-err-decay}, does not imply that the ``potential energy'' $G(x_{k+1})-G^*$ or the ``kinetic energy'' $\frac{\mu}{2} \norm{v_{k+1}-x^*}{\LL}^2$ decay monotonically by themselves. Corollary \ref{cor:PAGD-conv} only asserts exponential bounds. The same is true for the decay of the $\LLinv$--norm of the residuals $\norm{G'(y_k)}{\LLinv}$. In fact, the numerical illustrations of Section~\ref{sec:numerical} show that these quantities may oscillate. 
	\ermk
	\end{rmk}

	\begin{rmk}[matching convergence rates]
	\label{rmk:rate-match}
As discussed in Section~\ref{sec:algorithms}, in the case of $G$ being locally Lipschitz smooth, $\mu-$strongly convex, the (best) contraction factor for PGD is $(\frac{1-\rho}{1+\rho})^2$ while we have $1-\sqrt \rho$ for PAGD (see Theorem \ref{thm:exp-conv}), where we recall $\rho=\mu/L_B$ and $L_B>0$ is the Lipschitz smoothness constant on some appropriate invariant set $\mathfrak B$. It must be pointed out that this rate for PGD is achieved by choosing a ``particularly good'' step size that is only available to PGD: $s=\frac{2}{L_B+\mu}$ (see \citep[Theorem 2.1.15]{nesterov_2014}). More specifically, we have a contraction factor $1-s\frac{2\mu L_B}{L_B+\mu}$ for PGD provided $0<s\le \frac{2}{L_B+\mu}$. If one uses the step size $s=1/{L_B}$, then the contraction factor for PGD turns out to be	$\frac{1-\rho}{1+\rho}$. This choice makes it easier to see the rate match the continuous time model. Setting $p=2$ in \citep[SI (Supplement Information) Theorem H.2]{wibisono2016} we see that the gradient flow
$ %	\[
\dot X = -\LLinv G'(X)
$ %	\]
	has convergence rate 	$G(X(t))-G^*\le (G(X(0))-G^*)e^{-\mu t}$. However, using an estimate available to $\mu$--strongly convex functions (see \citep[Theorem 2.1.10 (2.1.19)]{nesterov_2014}), we can do better to get $G(X(t))-G^*\le (G(X(0))-G^*)e^{-2\mu t}$. % (Wibisono et al. did not use an optimal constant for the case of $p=2$ in the proof in order to incorporate a more general notion of strong convexity). 
	Then, we see that setting $t=sk$ and $s=1/L_B$ for the gradient flow, and assuming $\rho\ll  1$, the contraction factor can be approximated by 
	\[
	e^{-2\mu s}\approx 1-2\mu s =1- 2\rho, 	  
	\]
which is close to $\frac{1-\rho}{1+\rho}$. Similarly, setting $t=\sz k$ and $s=1/L_B$ in \eqref{est:Lyap-decay} and referring to Corollary \ref{cor:PAGD-conv}, we have the contraction factors
	\[
	  e^{-\sqrt {\mu s}}\approx 1-\sqrt{\mu s} =1- \sqrt{\rho}
	\]
  for the IVP \eqref{the-IVP}, which matches that of PAGD.
	\ermk
	\end{rmk}

\section{Numerical Experiments} \label{sec:numerical}

In this section, we carry out a series of numerical experiments aimed at illustrating the theory that we have developed. In all our examples, we approximate the solution to the nonlinear PDE \eqref{model-pb} by iteratively minimizing an energy related to this PDE. The approximate solution is computed using a pseudo-spectral method (see \citep{canuto2007spectral, MR2867779}), which was implemented in an in-house Matlab R2016a\copyright{} code. This pseudo-spectral code heavily uses the built-in \texttt{fft} and \texttt{ifft} Matlab internal routines to invert preconditioners and apply residuals.
	
The energy minimization is carried out with GD, AGD, PGD, or PAGD, where the algorithms terminates if one of the following cases is true:
\begin{enumerate}[(a)]
  \item the $\infty$--norm (when the true solution is unknown) or the $\LL_N$--norm (when the true solution is known) of the search direction is smaller than a certain tolerance, which we will call \emph{convergence};
  
  \item the norm being measured is larger than a certain upper tolerance, which we will call \emph{blow up};
  
  \item the number of iterations reaches a certain number, which we will call \emph{no convergence}.
\end{enumerate}
In the conditions above, we mean by ``search directions'' the residual if the scheme does not involve a preconditioner. If the scheme involves a (discrete) preconditioner $\LL_N$ (see \eqref{def:dis-precon} for definition), the search direction is the solution to $\LL_N s=r$, where $r$ is the residual. In all implementations, the initial guess is always zero.

\subsection{The continuous problem}
We approximate the solution to the following ``nonlocal" PDE:
\begin{equation}
(-\lap)^\alpha u +|u|^{p-2}u + t u=f	\quad \text{ in } \Omega=(0,1)^2 \subset\RR^2,
\label{model-pb}
\end{equation}			
supplemented with periodic boundary conditions, where $\alpha>0$, $p\ge 2$, and $t>0$. Here and in what follows, all functions are real-valued except for the exponential functions appearing in Fourier series and Fourier coefficients. The nonlocal operator $(-\lap)^\alpha$ is the \emph{spectral} fractional Laplacian, which is defined via Fourier series as 

For every $v\in L^2_{\per}(\Omega)$, we have that
$%\[
v(\x)=\sum_{\bfm\in\ZZ^2} \hat {v}_\bfm	e^{2\pi \fraki \bfm \cdot \x},
$ %	\]
where the equality is in the $L^2(\Omega)$--sense, $\x=(x,y)\in\overline{\Omega}$, $\fraki = \sqrt{-1}$, and 
$%	\[
\hat{v}_\bfm =\int_{\Omega} v(\x) e^{-2\pi \fraki \bfm\cdot\x} \diff \x ,\quad \bfm\in\ZZ^2.
$ %	\] 

Thus, we define
\[
(-\lap)^\alpha  v(\x)=\sum_{\bfm\in\ZZ^2} \left( 4\pi^2 |\bfm|^2 \right)^\alpha \hat{v}_\bfm	e^{2\pi \fraki \bfm\cdot\x},
\]
provided that the sum is finite.

With this definition at hand, it is not difficult to see that \eqref{model-pb} in its weak form, can be seen as the Euler-Lagrange equation for the functional
\begin{equation}
G(u)= \int_\Omega \left(	\half|(-\lap)^\frac{\alpha}{2} u|^2 +\oneover{p} 	|u|^p +\frac{t}{2} |u|^2-	f u \right) \diff \x,
\label{obj:model-G}
\end{equation}
over the space $\HH=H^\alpha_\per(\Omega)\cap L^p(\Omega)$. %, that is, the space of periodic, $p$--integrable functions, whose $\alpha$--order derivatives in the Fourier sense are also square integrable.  
It is well known that $\{e^{2\pi \fraki \mathbf{m} \cdot \mathbf{x}} \}_{\mathbf{m} \in \ZZ^2}$ is an orthonormal basis of $L^2_\per(\Omega)$. Then, $\HH$ can be equivalently defined via
\begin{equation}
\HH = \left\{v\in L^p_\per(\Omega) \middle| \sum_{\mathbf{m}\in \ZZ^2}|\mathbf{m}|^\alpha |\hat v_\mathbf{m}|^2 < \infty \right\}.
\nonumber %\label{def:H^1_per}
\end{equation}

The existence and uniqueness of a weak solution to \eqref{model-pb} is guaranteed for any $f\in L^{p'}_\per(\Omega)$, where $1/p+1/p'=1$, since, in this case, the energy is well-defined, strictly convex, and coercive.

For the space $\HH$ to possess a Hilbert structure, a restriction on $p$ must be imposed depending on $\alpha$. For ease of notation, let $(\cdot,\cdot)$ and $\|\cdot\|$ denote the $L^2(\Omega)$--inner product and $L^2(\Omega)$--norm respectively. A natural inner product on $H^\alpha_{\per}(\Omega)$ is given by 
$%	\begin{align}
(v,w)_{H^\alpha_{\per}(\Omega)}=((-\lap)^\frac{\alpha}{2} v, (-\lap)^\frac{\alpha}{2} w)+(v,w)	,
$ %	\end{align}

and its associated norm by $\norm{v}{H^\alpha_{\per}(\Omega)}=\sqrt{ (v,v)_{H^\alpha_{\per}(\Omega)}}$. The following is a standard Sobolev embedding result. For a proof, see,  e.g., \citep[Theorem 7.34]{adams2003sobolev}.

	\begin{prop}[Sobolev embedding]
	\label{prop:SobEmb}
Let $\alpha\in(0,1]$. For all $p\in[2,p^*]$ with $p^*=\frac{2}{1-\alpha}$ if $\alpha<1$ or $p\in[2,\infty)$ if $\alpha=1$, there exists $C_{emb}=C_{emb}(p,\alpha)>0$ such that, for all $v\in H^\alpha_{\per}(\Omega)$, 
	\begin{align}
	\norm{v}{L^p(\Omega)}\le C_{emb}\norm{v}{H^\alpha_{\per}(\Omega)}.
	\label{est:emb}
	\end{align}		
	\end{prop}

We introduce the preconditioner 
$ %\begin{align}
\LL u= (-\lap)^\alpha u +\nu u 	,
 %\label{obj:model-L}
$ %\end{align}	

where $\nu \geq0$ is a free parameter, which induces a inner product  
\begin{align}
\pairing{\LL u}{v}=\int_\Omega \left((-\lap)^\frac{\alpha}{2} u (-\lap)^\frac{\alpha}{2} v + \nu uv \right) \diff \x.	
\label{obj:model-L}
\end{align}

\begin{rmk}[notation]
	As it is clear from its definition, the Lipschitz constant of $G'$ depends on the norm being used. Thus, we will make a difference between the case with preconditioner and without it. $\Lhat$ denotes the Lipschitz constant with respect to the preconditioner-induced norm $\norm{\, \cdot \, }{\LL}$, while $L$ is the constant with respect to the original norm $\norm{\, \cdot \, }{\HH}$.
\ermk\end{rmk}

We investigate the properties of $G$ in the following result.

\begin{prop}[properties of $G$]
	\label{prop:propertiesG}
	Let $G$ be given by \eqref{obj:model-G} and the preconditioner $\LL$ by \eqref{obj:model-L}. Then, $G$ is strongly convex with respect to $\LL$--norm. If, in addition, $p$ satisfies the conditions of Proposition~\ref{prop:SobEmb}, then $G$ is locally Lipschitz smooth with respect to $\LL$--norm.
\end{prop}
\begin{proof}
	First, the action of $G'$ is characterized by the following: for $v,w \in \HH$,
	\begin{align*}
	\pairing{G'(v)}{w}=((-\lap)^\frac{\alpha}{2} v, (-\lap)^\frac{\alpha}{2} w)+t(v,w)+(|v|^{p-2}v,w)-(f,w).
	\end{align*}
	Note also that the following estimates hold, which are a special case of \citep[Lemma 2.1]{barrett93}: for $p>1$, there exist $C_{p1}, C_{p2}>0$, which depend only on $p$, such that for all $\xi,\eta\in\mathbb{R}$,
	\begin{align}
	||\xi|^{p-2}\xi-|\eta|^{p-2}\eta|&\le C_{p1} |\xi-\eta|(|\xi|+|\eta|)^{p-2},
	\label{est:p-upper}
	\\
	(|\xi|^{p-2}\xi-|\eta|^{p-2}\eta)(\xi-\eta)&\ge C_{p2} |\xi-\eta|^2(|\xi|+|\eta|)^{p-2}.
	\label{est:p-lower}
	\end{align} 	
	Thus, using \eqref{est:p-lower}
	\begin{align*}
	\quad\pairing{G'(v)-G'(w)}{v-w}
	\ge \norm{(-\lap)^\frac{\alpha}{2} (v-w)}{}^2+t\norm{v-w}{}^2\ge \muhat \norm{v-w}{\LL}^2,
	\nonumber
	\end{align*}
	where
	\begin{equation}
	\muhat=\min\{1,t/\nu\}
	\label{eq:muhat}.
	\end{equation}
	Observe that this holds without referring to Sobolev embedding. Note also that this implies the coercivity of $G$ with respect to $\LL$--norm, that is, $\lim_{\norm{v}{\LL}\goesto\infty}{G(v)}=\infty$. 
	
	Next, thanks to coercivity, for any bounded, convex set $B\subset\HH$ there exists $M_B\in\RR$ such that $B \subset \left\{x\in\HH \ \middle| \ G(x)\le M_B \right\}$. Hence, for each $v\in B$, using Cauchy-Schwarz  inequality and Young's inequality, it follows that there exists $\varepsilon>0$ such that 
	\begin{align}
	M_B&\ge G(v)=\half \norm{(-\lap)^\frac{\alpha}{2} v}{}^2 +\oneover{p} \norm{v}{L^p(\Omega)}^p+\frac{t}{2}\norm{v}{}^2-(f,v)
	\nonumber
	\\
	&\ge \half \norm{(-\lap)^\frac{\alpha}{2} v}{}^2 +\oneover{p} \norm{v}{L^p(\Omega)}^p+\frac{t}{4}\norm{v}{}^2-\frac{1}{2\varepsilon} \norm{f}{}^2.
	\label{est:Lp-nrm-bd}	
	\end{align}
	Rearranging this, we see that there exists $C_{f,t,p,B}>0$ such that 
	\begin{align}
	\norm{v}{L^p(\Omega)}\le C_{f,t,p,B} \quad \forall v\in B.
	\label{est:B-Lp-engy-bd}
	\end{align}
	
	 On the other hand, using \eqref{est:p-upper}, H\"older's inequality, and \eqref{est:B-Lp-engy-bd}, we have, for all $v,w\in B$,
	\begin{align*}
	&\quad\pairing{G'(v)-G'(w)}{v-w}
	\nonumber
	\\
	&=\norm{(-\lap)^\frac{\alpha}{2} (v-w)}{}^2+t\norm{v-w}{}^2+(|v|^{p-2}v-|w|^{p-2}w,v-w)
	\\
	&\le \norm{(-\lap)^\frac{\alpha}{2} (v-w)}{}^2+t\norm{v-w}{}^2 +C_{p1}\int_\Omega |v-w|^2(|v|+|w|)^{p-2} \diff \x
	\\
	&\le \norm{(-\lap)^\frac{\alpha}{2} (v-w)}{}^2+t\norm{v-w}{}^2 +C_{p3} \|v-w\|^2_{L^p(\Omega)} \left(\|v\|_{L^p(\Omega)}^{p-2}+\|w\|_{L^p(\Omega)}^{p-2} \right)
	\\
	&\le \norm{(-\lap)^\frac{\alpha}{2} (v-w)}{}^2+t\norm{v-w}{}^2 +2C_{f,t,p,B}^{p-2} C_{p3}\|v-w\|^2_{L^p(\Omega)},
	\end{align*}
	where $C_{p3}>0$ is a constant reflecting the equivalence between $(|v|+|w|)^{p-2}$ and $|v|^{p-2}+|w|^{p-2}$.
	
	Finally, owing to the restriction on $p$, Proposition~\ref{prop:SobEmb} guarantees that
	\[
	\|v-w\|^2_{L^p(\Omega)} \leq C_{emb}^2 \| v - w \|_{H^\alpha_\per(\Omega)}^2,
	\]
	so that
	\[
	\pairing{G'(v)-G'(w)}{v-w} \le \Lhat_B\norm{(v-w)}{\LL}^2 ,
	\]	
	with $\Lhat_B=\max\{1, t/\nu, 2C_{f,t,p,B}^{p-2}C_{p3}C_{emb}^2\}$. 
\end{proof}

\begin{rmk}[strong Lipschitz smoothness]
The proof of Proposition~\ref{prop:propertiesG} can be easily modified to show that $G$ is locally Lipschitz smooth in the strong sense, i.e., \eqref{def:Lip-str} holds.
\ermk
\end{rmk}

\subsection{Discretization}
We discretize the model problem \eqref{model-pb} by introducing a uniform grid of points. To simplify the presentation, we choose $N\in\NN$ with $N=2K+1$ for some integer $K\ge 1$. (The details for the  case that $N$ is even are only slightly more complicated.) Define $h=1/N$, and introduce the grid domain
$ %\[
\gm= \left\{(x_\ell, y_m)\in[0,1]^2 \ \middle| \ x_\ell = \ell h, \ y_m = mh, \ 0\le \ell, m \le N \right\}.
$ %\]
For ease of notation, let us introduce
$ %\begin{align*}
\mathbb{N}^2_{N} = \left\{ \bfm =(m_1,m_2) \in\ZZ^2 \ \middle| \ 1\le m_1,m_2\le N \right\}
$ and 
$
\ZZ^2_K = \left\{ \bfr =(r_1,r_2) \in\ZZ^2 \ \middle| \ -K\le r_1,r_2\le K\right\}
$, %\end{align*}
Then, for $\bfm \in \mathbb{N}^2_{N}$, we can denote $\x_\bfm=(x_{m_1},y_{m_2})\in \gm$. This notation must not be confused with that of the iterates of PAGD.
Define the space of periodic grid functions
\begin{align}
	\gridfn = \left\{v_N:\gm\goesto\RR \ \middle| \ v_N(0,hm)=v_N(hN,hm),\ v_N(h\ell,0)=v_N(h\ell,hN), \right. \quad& 
	\nonumber
	\\
	\left.0\leq m,\ell \leq N \right\},&
	\label{obj:grid-fn}
\end{align}
endowed with the $L^2_N$--inner product
$ %\begin{equation}
	(v_N,w_N)_N = h^2 \sum_{\bfm\in \mathbb{N}^2_{N}} v_N(\x_\bfm)w_N(\x_\bfm).
%	\label{obj:disc-L2-ip}
$ %\end{equation}
More generally, for $p\ge 1$, define
$ %	\begin{align}
\norm{w_N}{N,p}=\left( h^2\sum_{\bfm\in\mathbb{N}^2_{N}} |w_N(\x_\bfm)|^p \right)^\frac{1}{p} .	
%	\nonumber
$ %	\end{align}
Given $w_N\in\gridfn$, its \emph{discrete Fourier transform} (DFT) is 
\[ 
\hat w_K(\bfr)=h^2\sum_{\bfs\in\mathbb{N}^2_{N}} w_N(\x_\bfs) e^{-2\pi \fraki \bfr\cdot\x_\bfs},
\quad \bfr \in \ZZ^2_K.
\]
The \emph{discrete fractional Laplacian} $(-\lap_N)^\alpha :\gridfn\goesto\gridfn$ is defined by
\begin{equation}
[(-\lap_N)^\alpha w_N](\x_\bfm)=\sum_{\bfr\in\ZZ^2_K} (4\pi^2 |\bfr|^2)^\alpha \hat w_K(\bfr) e^{2\pi \fraki \bfr\cdot\x_{\bfm}}.
	\label{def:frac-Lap}
\end{equation}
Finally, for $v_N, w_N\in\HH_N$, the $H^\alpha_N$--inner product is given by
$	%\begin{align}
\iprd{v_N}{w_N}_{H^\alpha_N} = (v_N,w_N)_N+((-\lap_N)^\frac{\alpha}{2}v_N, (-\lap_N)^\frac{\alpha}{2}w_N)_N,	
%	\label{obj:disc-Halpha-ip}
$	%\end{align}
and $\norm{w_N}{H^\alpha_N}=\sqrt{(w_N,w_N)_{H^\alpha_N}}$.

We comment that there are, at least, three different natural choices for the underlying inner product for $\HH_N$: the $L^2_N$--inner product, the $H_N^\alpha$--inner product, and the $\LL_N$--inner product, \eqref{def:dis-precon}. We will choose the first option, i.e., the $L^2_N$--inner product for several reasons. To begin with, this is the way numerical experiments are usually done if no special distinction is made between representers of the residual with respect to multiple inner products.  In addition, this illustrates the effect of preconditioning more vividly. For example, if we adopt $H^\alpha_N$--inner product, this leads to a preconditioned scheme in disguise: finding a representer of the residual with respect to this inner product is equivalent to using $\LL_N$--inner product \eqref{def:dis-precon} with $\nu_N=1$. On the other hand, $L^2_N$--inner product leads to truly \emph{non-preconditioned} schemes such as GD or AGD.

After having introduced all this notation, we can write our discrete problem as: given $f_N\in\gridfn$, find $u_N\in\gridfn$ such that 
\begin{equation}
(-\lap_N)^\alpha u_N	+|u_N|^{p-2}u_N +tu_N=f_N.
\label{eq:disc-PDE}
\end{equation}
In this problem, $f_N\in\gridfn$ is some approximation of the problem data $f$. For example, if $f$ is continuous, $f_N(\x_\bfm)=f(\x_\bfm)$ is a natural option, and if $f$ is only an $L^2(\Omega)-$function, then the sampling at the nodes of the $L^2(\Omega)-$projection of $f$ onto ${\mathcal P}_K$, the trigonometric polynomial of degree at most $K$, is natural although these two may not agree even if one starts with the same continuous function $f$. In fact, the difference between these two possibilities is very small if $f$ is smooth and its derivatives are periodic (see \citep[pp. 44---45]{canuto2007spectral}).

Our discrete problem has a similar energy structure to the continuous problem. It is the Euler-Lagrange equation of the following functional 
\begin{equation}
G_N(v_N) = \half\norm{(-\lap_N)^\frac{\alpha}{2} v_N}{N}^2+\oneover{p}\norm{v_N}{N,p}^2+\frac{t}{2}\norm{v_N}{N}^2-\iprd{f_N}{v_N}_{N}.
\label{obj:discrete-egy}
\end{equation}	

We introduce a (discrete) preconditioner
\begin{align}
\LL_N=(-\lap_N)^\alpha +\nu_N \identity_N,
\label{def:dis-precon}
\end{align}
where $\nu_N>0$ and $\identity_N:\gridfn\goesto\gridfn$ is the identity map. The parameter $\nu_N>0$ will be determined later. This preconditioner induces an inner product on $\gridfn$ given by
\begin{equation}
(v_N,w_N)_{\LL_N}= \nu_N (v_N,w_N)_{N}+\left((-\lap_N)^\frac{\alpha}{2} v_N,(-\lap_N)^\frac{\alpha}{2} w_N\right)_{N} ,
	\label{obj:disc-L-ip}
\end{equation}
and an associated norm $\norm{v_N}{\LL_N}=\sqrt{(v_N,v_N)_{\LL_N}}$. It is desirable that the convergence of our scheme does not deteriorate as we refine the grid points. We can ensure this under a certain restriction on $p$. The following proposition provides an important tool for that purpose.

\begin{prop}[discrete Sobolev embedding]
	\label{prop:dis-Sob-emb}
	Let $\alpha\in (0,1]$. For all $p\in[2, p^*]$ with $p^*=\frac{2}{1-\alpha}$ if $\alpha<1$ or for all $p\in[2,\infty)$ if $\alpha=1$, there exists a constant $C_{p, \alpha}>0$ such that, for all $v_N\in\HH_N$,
	\begin{align}
	\norm{v_N}{N,p}\le C_{p,\alpha}\norm{v_N}{H^\alpha_N}.
	\end{align}		
	$C_{p,\alpha}$ is independent of $v_N$ and $N$.	
\end{prop}
\begin{proof}
		Note that $\norm{v_N}{N,p}\le C\norm{v}{L^p(\Omega)}$ for all $v_N\in\HH_N$; see, for instance, \citep[Lemma 2.48]{jovanovic2013analysis}, where $v$ is the unique trigonometric polynomial of degree less than or equal to $K$ interpolating $v_N$ and $C>0$ depends only on the dimension of $\Omega$.  Also, the following Parseval's identity holds in the fractional setting $%\begin{align}
		\norm{v_N}{H^\alpha_N}=\norm{v}{H^\alpha_{\per}(\Omega)}
		%\label{eql:parseval}
		$. %\end{align}		
		In conjunction with the Sobolev embedding at the continuous level, \eqref{est:emb}, we have, for any $v_N\in\HH_N$, 
	\[
	\norm{v_N}{N,p}\le C\norm{v}{L^p(\Omega)} \le CC_{p,\alpha}\norm{v}{H^\alpha_{\per}(\Omega)}=CC_{p,\alpha}\norm{v_N}{H^\alpha_N}.
	\qedhere
	\]
\end{proof}

The following result addresses dimension-independence of the (inverse) condition number as well as the structure of $G_N$ that is needed to apply the theory we have developed in Section~\ref{sec:PAGD-conv}. Note that, in the following statement, the sublevel sets of $G_N$ provide a compatible way to describe bounded, convex sets when we consider multiple resolutions since, strictly speaking, for different values of $N$, functions in $\HH_N$ may not be directly comparable. %We omit its proof since it is parallel to Proposition \ref{prop:propertiesG}

\begin{thm}[properties of $G_N$]
	\label{thm:prpty-G_N}
	Let the space of grid functions $\HH_N$ be given by \eqref{obj:grid-fn} and the preconditioner $\LL_N$ by \eqref{def:dis-precon}. Then, the energy functional $G_N:\HH_N\goesto\RR$ defined by \eqref{obj:discrete-egy} is strongly convex and locally Lipschitz smooth with respect to the $\LL_N$--norm. Moreover, the strong convexity constant $\hat \mu_N$ is independent of $N$. Suppose, in addition, that $p$ satisfies the conditions of Proposition~\ref{prop:dis-Sob-emb} and that $f_N$ is defined in a stable manner when we pose the discrete problem \eqref{eq:disc-PDE}, i.e., there exists $C>0$ independent of $N$ such that
	\[
    \norm{f_N}{N,2}\le C\norm{f}{L^2(\Omega)}.
	\]
 Then, the local Lipschitz smoothness constant $\hat L_N$ is also independent of $N$ in the sense that, for each $M\in\RR$, $G_N$ is $\hat L_N$--Lipschitz smooth on the sublevel set
 \[
   \left\{v_N\in\HH_N \ \middle| \ G(v_N)\le M \right\}
 \]
 with $\hat L_N$ independent of $N$. Consequently, the (inverse) condition number $\hat \mu_N/\hat L_N$ with respect to the $\LL_N$--norm stays away from 0 as $N\goesto\infty$.
\end{thm}
\begin{proof}
	The proof of the strong convexity is parallel to that of Proposition \ref{prop:propertiesG}. The proof of local Lipschitz smoothness is also parallel, but we need the assumed stability of $\norm{f_N}{N,2}$ to proceed from \eqref{est:Lp-nrm-bd} to \eqref{est:B-Lp-engy-bd}. Finally, to complete the proof, we simply replace the embedding constant $C_{emb}$ with its discrete counterpart $C_{p,\alpha}$ as given in Proposition \ref{prop:dis-Sob-emb}.  
\end{proof}

Since the (inverse) condition number $\muhat_N/\Lhat_N$ governs the rate of convergence of PAGD (Corollary \ref{cor:PAGD-conv}), the previous theorem guarantees that one can achieve the same rate of convergence even if we refine the number of grid points $N\goesto\infty$. However, this is true in terms of the number of iterations, but the wall clock time will take longer as the refinement is conducted.

Being in finite dimensions, $G_N$ is also strongly convex, and locally Lipschitz smooth with respect to \emph{any norm}, for instance $\norm{\, \cdot \, }{N}$. The constants in this case, however, are different and depend on the dimension of $\gridfn$, which obviously depends on the number of grid points, and thus on $N$. We label them $\mu_N$ and $L_N$ to distinguish them from the dimension-independent constants $\muhat_N$ and $\Lhat_N$ respectively.

\subsection{A problem with a manufactured solution}
In this first experiment, we solve \eqref{eq:disc-PDE} by minimizing the energy \eqref{obj:discrete-egy}. To compute the errors and energies, the following manufactured solution is used
\[
u_N(\x_\bfm)= \exp\left(\sin 2\pi\left(x_{m_1}-\frac{1}{4} \right) +\sin 4\pi\left(y_{m_2}-\frac{3}{8}\right)\right).
\]

We set $\alpha=0.5$, $p=4$, $t=1$, $N=64$, and found, experimentally, that the values $\nu_N=1.2$, $\mu_N=1$ are optimal, while we set $\muhat_N=5/6=\min \{ 1, t/\nu_N\}$ in view of \eqref{eq:muhat}. 
To specify step sizes, recall the step size rules that theoretically guarantee convergence (see Section~\ref{sec:algorithms}):
 $s=2/(L_N+\mu_N)$ for GD, $s=1/L_N$ for AGD, $s=2/(\Lhat_N+\muhat_N)$  for PGD, and $s=1/\Lhat_N$ for PAGD. Step sizes are set by these relations with $L_N=500$ and $\Lhat_N=20$, which are also experimentally proved to be optimal. However, it must be noted that this is just a way of setting step sizes. We do not really know neither whether the values for $L_N$ or $\Lhat_N$ are the Lipschitz constants of the corresponding energy functionals nor whether the aforementioned step size rules give the optimal results even if we knew the Lipschitz constants. In fact, our last experiment suggests that larger step sizes than what is theoretically proven seem to work.
 
\begin{figure}
	\centering
	\subfloat[Objective, $G_N(x_k)-G_N(u_N)$, plot of GD, AGD, PGD, and PAGD.]{\label{figsub:known-all-egy} \includegraphics[width=0.5\linewidth]{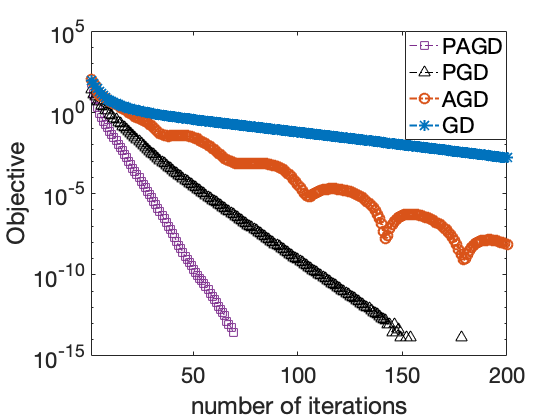}} 
\subfloat[$\LL_N$--norm of errors that are generated by GD, AGD, PGD, and PAGD.]{\label{figsub:known-all-err}\includegraphics[width=0.5\linewidth]{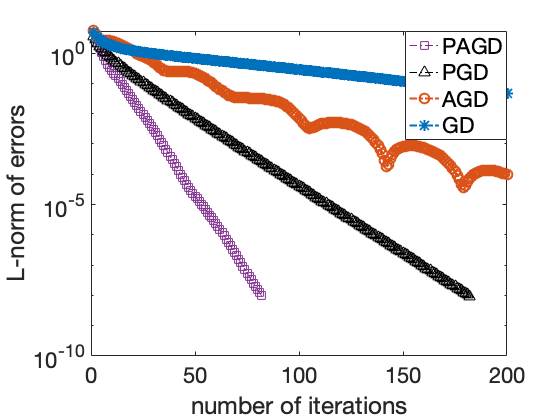}}
\\
	\subfloat[Potential, kinetic, and total energy plot of GD and AGD.]{\label{figsub:known-nonpre-egy} \includegraphics[width=0.5\linewidth]{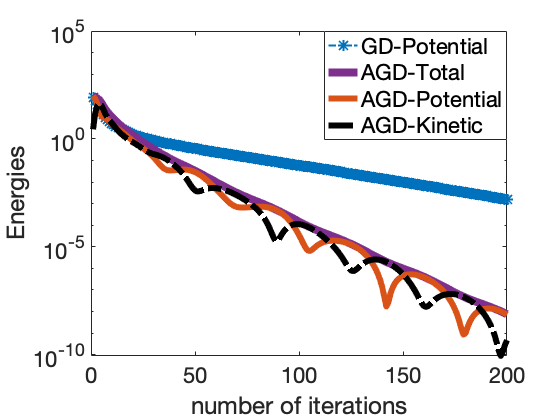}}
	\caption{Objective, error, and energy decay plots for GD, AGD, PGD, and PAGD. They are implemented to solve \eqref{eq:disc-PDE} by minimizing \eqref{obj:discrete-egy} ($\alpha=0.5$, $p=4$, $t=1$, $N=64$, $\nu_N=1.2$, $L_N=500$, $\mu_N=1$, $\Lhat_N=20$, $\muhat_N=5/6=\min \{ 1, t/\nu_N\}$, and step sizes are set via $s=2/(L_N+\mu_N)$ for GD, $s=1/L_N$ for AGD, $s=2/(\Lhat_N+\muhat_N)$  for PGD, and $s=1/\Lhat_N$ for PAGD). The vertical axes (logarithmic scale) show the value of the objective, $\LL_N$--norm of errors, or various energies while the horizontal axis (linear scale) shows the number of iterations.}
	\label{fig:Exp1}	
\end{figure}	

Figure~\ref{fig:Exp1} shows the performance of GD, AGD, PGD, and PAGD when used to solve \eqref{eq:disc-PDE} by minimizing \eqref{obj:discrete-egy},  where the data is as described above. The stopping criteria take the following parameters: the tolerance is $10^{-8}$, the upper tolerance is $10^{10}$, and the maximum number of iterations is 200.

Figure~\ref{fig:Exp1} \subref{figsub:known-all-egy} shows the decay of the objective, $G_N(x_k)-G_N(u_N)$ which is, up to a constant, the same as the decay of the potential energy, for all four schemes. Here $k$ is the number of iterations. 
Figure \ref{fig:Exp1} \subref{figsub:known-all-err} shows the decay of the $\LL_N$--norm of the errors. Notice that PAGD performs significantly better than all the other methods.   

Figure~\ref{fig:Exp1} \subref{figsub:known-nonpre-egy} shows the performance of GD and AGD. Since these schemes do not involve a preconditioner, the corresponding total energy is defined by
\begin{align}
E_N(x_k,v_k) = \oneover{\eta_N} (G_N(x_k)-G_N(u_N))+ \frac{\eta_N}{2} \norm{v_k-u_N}{N}^2,
\label{obj:disc-total-egy-non}
\end{align}
where $k$ is the number of iterations and $\eta_N=\sqrt{\mu_N}$. The first and second terms can be understood as potential and kinetic energy respectively.
Figure~\ref{fig:Exp1} \subref{figsub:known-nonpre-egy} shows the decay of various energies for nonpreconditioned schemes. This figure better illustrates our analysis of the previous section than the preconditioned ones since they converge slower. As expected, AGD performs substantially better than GD. The total energy of AGD decreases steadily and exponentially fast. Notice that the vertical axis is in logarithmic scale. This matches what is predicted by the theory in Theorem~\ref{thm:exp-conv}. Observe also that the potential and kinetic energies of AGD, by themselves, oscillate; see Remark~\ref{rmk:each-egy-osc}. The physical analogy for AGD described in Remark~\ref{rmk:ODE-phys-interp} is clear from this picture. A fraction of the potential energy is converted to kinetic energy and they fluctuate as the mechanical system converges to equilibrium.

\subsection{A problem where the solution is unknown}
In this second experiment we, again, solve \eqref{eq:disc-PDE} by minimizing the energy \eqref{obj:discrete-egy}. The discrete right hand side $f_N$ is given by
\begin{align}
f_N(\x_\bfm)=\exp\left( \sin 2\pi (x_{m_1}-0.25) +\sin 2\pi (y_{m_2}-0.25) \right).
\label{obj:forcing-exp}
\end{align}
The parameters of the PDE are set to $\alpha=0.5$, $p=10$, and $t=1$. Observe that for these values of $\alpha$ and $p$ we do not have that $H^\alpha_{\per}(\Omega)\emb L^p(\Omega)$. We found, experimentally, that the choice $\nu_N=0.9$ is optimal for the preconditioner. We also set $\mu_N=1$ and $\muhat_N=1=\min \{1, t/\nu_N\}$ in view of \eqref{eq:muhat} as before. Step sizes are set in the same way as in the previous experiment with $L_N=300$ or $3000$, and  $\Lhat_N=9$. The values of $\mu_N$, $L_N$, and $\Lhat_N$ were experimentally found to be optimal except for $L_N=3000$. That is, they yield the best convergence rate with all other parameters being fixed. A more detailed explanation about $L_N=3000$ will follow. Two different degrees of resolution are used to show the dimension dependence of nonpreconditioned schemes. The stopping criterion parameters are as before.

\begin{figure}
	\subfloat[$\infty$--norm of search direction ($N=64$, $L_N=300$).]{\label{subfig:L700-N64}\includegraphics[width=0.5\linewidth]{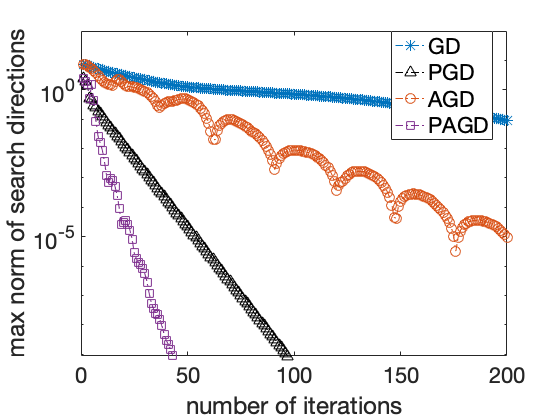}}
	\subfloat[$\infty$--norm of search direction ($N=512$, $L_N=300$).]{\label{subfig:L700-N512}\includegraphics[width=0.5\linewidth]{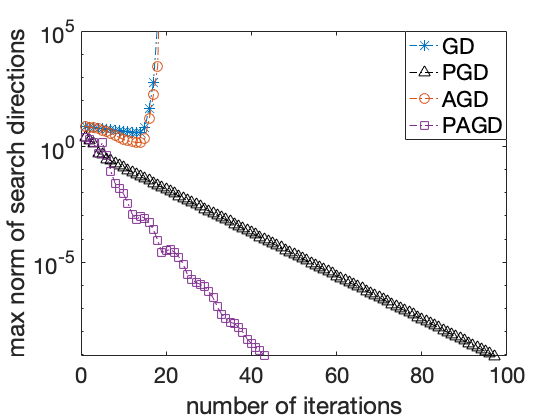}}
\\
	\subfloat[$\infty$--norm of search direction ($N=512$, $L_N=3000$).]{\label{subfig:L7000-N512}\includegraphics[width=0.5\linewidth]{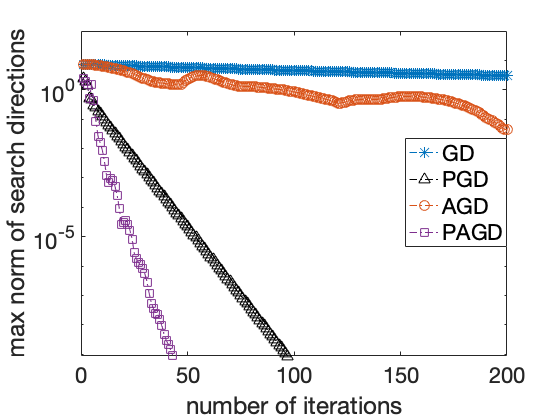}}
	\caption{$\infty$--norm plots of the search directions for GD, AGD, PGD, and PAGD. They are implemented to solve \eqref{eq:disc-PDE} by minimizing \eqref{obj:discrete-egy} with varying resolutions $N\in\{64, 512\}$ and varying step sizes for GD and AGD; $s=2/(L_N+\mu_N)$ for GD and $s=1/L_N$ for AGD with $L_N\in \{300, 3000\}$. The other parameters are set to $\alpha=0.5$, $p=10$, $t=1$, $\nu_N=0.9$, $\mu_N=1$, $\muhat_N=1=\min\{1, t/\nu_N\}$, $s=2/(\Lhat_N+\muhat_N)$  for PGD, and $s=1/\Lhat_N$ for PAGD with $\Lhat_N=9$. The horizontal axis (linear scale) represents the number of iterations. The vertical axis (logarithmic scale) represents $\infty$--norm of the search directions.}
	\label{fig:Exp2-maxnorm}
\end{figure}

Figure~\ref{fig:Exp2-maxnorm} shows the $\infty$--norm of the search directions for GD, AGD, PGD, and PAGD with varying degrees of resolution and with two different step sizes for GD and AGD, which are determined by the same step size rules as in the previous experiment with $L_N\in\{300,3000\}$. In Figure~\ref{fig:Exp2-maxnorm} \subref{subfig:L700-N64}, we observe a similar performance as in Figure~\ref{fig:Exp1}. Recall that we do not have Sobolev embedding. Thus, one can expect the Lipschitz constant $L_N$, hence the step size, to depend on the number of grid points. In fact, theory predicts that even $\Lhat_N$ depends on it. However, for $\Lhat_N$, such dependence is not observed within the range of $N$ that we have chosen. We see that the step size for convergence indeed depends on $N$ in Figure~\ref{fig:Exp2-maxnorm} \subref{subfig:L700-N512}. As we increase the resolution of the grid from $N=64$ to $N=512$, nonpreconditioned schemes become unstable. Figure~\ref{fig:Exp2-maxnorm} \subref{subfig:L7000-N512} shows that the stability of GD and AGD is recovered after $L_N$ is increased from $300$ to $3000$, which amounts to decreasing the step size to roughly a tenth of the old one. ($L_N=3000$ is not optimally chosen).

\begin{figure}
	\subfloat[Number of iteration to reach a tolerance ($L_N=300$, $\Lhat_N=9$).]{\label{subfig:L700}\includegraphics[width=0.5\linewidth]{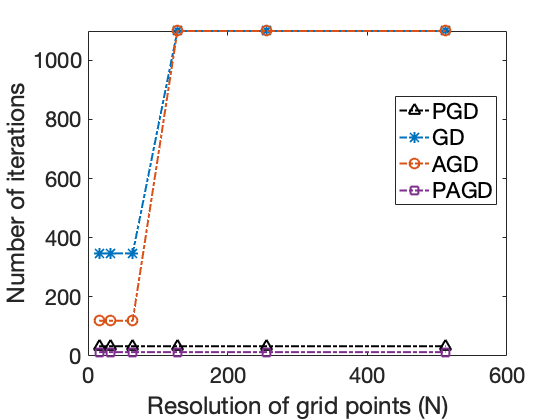}}
	\subfloat[Number of iteration to reach a tolerance ($L_N=3000$, $\Lhat_N=9$).]{\label{subfig:L7000}\includegraphics[width=0.5\linewidth]{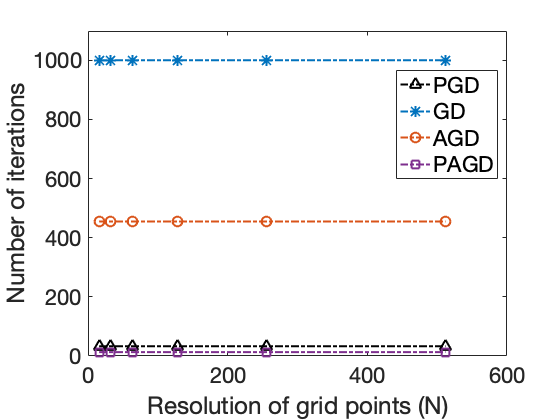}}
	
	\caption{Number of iterations for $\infty$--norm of the search directions to reach the tolerance $10^{-3}$ for GD, AGD, PGD, and PAGD. They are implemented to solve \eqref{eq:disc-PDE} by minimizing \eqref{obj:discrete-egy} with varying resolutions $N=16, 32, 64, 128, 256, 512$ and varying $L_N=300, 3000$ ($\alpha=0.5$, $p=10$, $t=1$, $\nu_N=0.9$, $\mu_N=1$, $L_N$ as indicated in the subfigures, $\muhat_N=1=\min\{1,t/\nu_N\}$, $\Lhat_N=9$). The horizontal axis represents the degrees of resolution, $N$. The vertical axis represents the minimum of the number of iterations for the $\infty$--norm of the search directions to reach the tolerance $10^{-3}$ (convergence) or 1000 iterations. The number of iterations being 1100 means that the $\infty$--norm of the search directions have reached the upper tolerance $10^8$ (blow up).}
	\label{fig:Exp2-N-dep}
\end{figure}

 Figure~\ref{fig:Exp2-N-dep} shows the dependence of $L_N$, hence the step size, on the number of grid points with the same experiment. However, here we use different tolerances and a different maximum number of iterations to best illustrate the dependence. For $N\in \{16, 32, 64, 128, 256, 512\}$, Figure~\ref{fig:Exp2-N-dep} records the number of iterations for $\infty$--norm of the search direction generated by each scheme to reach a tolerance $10^{-3}$  (``convergence'') or the maximum number of iterations, which is set to be 1000, if it does not reach the tolerance (``no convergence''). If the $\infty$--norm of the search direction reaches an upper tolerance $10^{8}$, the algorithm records the number of iteration taken as $1100$, which indicates ``blowing up.'' Figure~\ref{fig:Exp2-N-dep} \subref{subfig:L700} shows when the step sizes of the nonpreconditioned schemes correspond to $L_N=300$ and those of the preconditioned ones correspond to $\Lhat_N=9$. GD and AGD converge until $N=64$. However, they become unstable for $N\ge 128$. Figure~\ref{fig:Exp2-N-dep} \subref{subfig:L7000} shows the same experiment with smaller step sizes, which correspond to $L_N=3000$. In this case, we recover the stability of GD and AGD.

\subsection{A comparison between PGD and PAGD}
\label{sub:PGDvsPAGD}

\begin{table}[]
	\centering
	\begin{tabular}{c|c|c|c|c|c|c|}
		\cline{2-7}
		& \multicolumn{3}{c|}{PGD} & \multicolumn{3}{c|}{PAGD} \\ \hline
		\multicolumn{1}{|c|}{$\alpha$} & \# iterations   & $\nu_N$    & step size     & \# iterations   & $\nu_N$    & step size      \\ \hline
		\multicolumn{1}{|c|}{0.1}   & 64       & 1.0   & 0.20  & 38       & 0.9   & 0.14   \\ \hline
		\multicolumn{1}{|c|}{0.2}   & 50       & 1.1   & 0.25  & 32       & 1.0   & 0.18   \\ \hline
		\multicolumn{1}{|c|}{0.3}   & 39       & 1.2   & 0.31  & 29       & 1.1   & 0.22   \\ \hline
		\multicolumn{1}{|c|}{0.4}   & 29       & 2.6   & 0.57  & 26       & 1.2   & 0.26   \\ \hline
		\multicolumn{1}{|c|}{0.5}   & 22       & 2.8   & 0.66  & 24       & 1.3   & 0.30   \\ \hline
		\multicolumn{1}{|c|}{0.6}   & 16       & 4.1   & 0.97  & 20       & 5.5   & 0.83   \\ \hline
		\multicolumn{1}{|c|}{0.7}   & 13       & 3.4   & 0.90  & 17       & 5.2   & 0.91   \\ \hline
		\multicolumn{1}{|c|}{0.8}   & 11       & 4.6   & 1.04  & 15       & 4.2   & 0.88   \\ \hline
		\multicolumn{1}{|c|}{0.9}   & 12       & 3.8   & 0.89  & 12       & 5.0   & 0.96   \\ \hline
		\multicolumn{1}{|c|}{1.0}     & 10       & 4.0   & 0.95  & 12       & 4.3   & 0.92   \\ \hline
		\multicolumn{1}{|c|}{1.5}   & 9        & 4.5   & 0.97  & 11       & 4.5   & 0.97   \\ \hline
		\multicolumn{1}{|c|}{2.0}     & 8        & 4.8   & 1.03  & 10       & 4.5   & 0.96   \\ \hline
		\multicolumn{1}{|c|}{2.5}   & 8        & 4.1   & 0.88  & 9        & 4.2   & 0.90   \\ \hline
		\multicolumn{1}{|c|}{3.0}     & 8        & 4.1   & 0.88  & 9        & 4.2   & 0.90   \\ \hline
	\end{tabular}
	\caption{The minimal number of iterations needed for the $\infty$--norm of the search direction of PGD and PAGD to reach a tolerance of $10^{-9}$ and the values of $\nu_N$ and $s$ (step size) that led to the minimum iterations for a range of values of $\alpha$. They are implemented to solve \eqref{eq:disc-PDE} by minimizing the energy \eqref{obj:discrete-egy}. $N=64$, $\alpha\in\{0.1j\;|\;j=1,2,3,\cdots,10\} \cup\{1.5,2.0,2.5,3.0\}\subset(0,3]$, $p=6$, $t=1$, $\muhat_N=\min \{1, t/\nu_N\}$, and $f_N$ is given by \eqref{obj:forcing-exp}. For each value of $\alpha$, we consider $\nu_N\in\{0.1j\;|\;j=1,2,3,\cdots,100\} \subset (0,10]$ and $s\in\{0.01j\;|\;j=1,2,3,\cdots,200\} \subset (0,2]$. Among the possible 20,000 possible combinations of $\nu_N$ and $s$, we display the values that give the minimal number of iterations.}
	\label{tab:alpha-vs-iter}
\end{table}

In this final collection of experiments, we aim at comparing the performance of PGD and PAGD in different scenarios. To do so, we solve the discrete problem \eqref{eq:disc-PDE} by minimizing the energy \eqref{obj:discrete-egy} with the right hand side given by \eqref{obj:forcing-exp} as before. The problem parameters are set as $\alpha\in\{0.1j\;|\;j=1,2,3,\cdots,10\} \cup\{1.5,2.0,2.5,3.0\}\subset(0,3]$, $p=6$, and $t=1$. We set $N=64$ and $\muhat_N=\min \{1, t/\nu_N\}$. Then, for each value of $\alpha$ (column 1 of Table \ref{tab:alpha-vs-iter}), PGD and PAGD are applied with $\nu_N\in \{0.1j\;|\;j=1,2,3,\cdots,100\} \subset (0,10]$ and the step size $s\in\{0.01j\;|\;j=1,2,3,\cdots,200\} \subset(0,2]$. Observe that neither Algorithm \ref{alg:PGD} nor Algorithm \ref{alg:PAGD} require knowledge of $\Lhat_N$ a priori. Instead, we directly set the step size in this last experiment. Among these 20,000 possible values of $\nu_N$ and $s$, the minimal number of iterations for the $\infty$--norm of the search direction generated by PGD and PAGD to reach a tolerance of $10^{-9}$ (convergence) is recorded (column 2 and column 5 of Table \ref{tab:alpha-vs-iter}, respectively). A pair of values, $\nu_N$ and $s$, that led to the minimal number of iterations is also recorded (columns 3 and 4 of Table \ref{tab:alpha-vs-iter} for PGD and columns 6 and 7 of Table \ref{tab:alpha-vs-iter} for PAGD). There can be multiple such pairs. If this is the case, the pair $(\nu_N,s)$ that comes the first in the lexicographic order is recorded.

As we can see from Table \ref{tab:alpha-vs-iter}, for the nonlocal PDE \eqref{eq:disc-PDE} with small $\alpha$ ($\alpha=0.1, 0.2, 0.3, 0.4$), PAGD performs better than PGD when they are implemented with their own best pair of parameters $\nu_N$ and $s$ among those pairs that were considered. In particular, in the cases of $\alpha=0.1, 0.2, 0.3$, PAGD outperforms PGD while the best values of $\nu_N$ for the two schemes are similar, hence directly comparing their performances roughly make sense. An interesting thing, however, is that  one cannot say that PAGD is \emph{always} better than PGD. In fact, for the remaining values of $\alpha$, PGD takes fewer iterations to converge in the aforementioned sense than PAGD provided they are equipped with their ``best" parameters for each method. It must be noted that this result does not contradict our theory. The theory only tells us some upper bounds about the rate of convergence of the two schemes within a certain range of step size when they involve the same preconditioner. It does not explain what happens outside of that. The result provided here perhaps illustrates the latter case. In any case, we can see an improvement in the convergence of PAGD compared to PGD for ``harder" problems (\eqref{eq:disc-PDE} with small $\alpha$), where a stronger nonlocality is involved.

	\section*{Acknowledgments}
SMW acknowledges partial financial support from NSF-DMS 1719854.
AJS has been partially supported by NSF-DMS 1720123.

	\bibliographystyle{plain}
\bibliography{Preconditioned_Nesterov}

	\appendix

	\section{An IVP as the Limit of the PAGD Method}
	\label{app-PAGD-to-IVP}	

\subsection{Derivation of the ODE}
\label{sub:DeriveODE}

Let us start with the same approach as in \citep{su}. We assume, as an \emph{ansatz}, that PAGD is a discretization of an ODE, which has a solution $X:[0,\infty)\goesto \HH$, which we often call a \emph{trajectory}. We also assume that $X$ is smooth enough, e.g., twice continuously differentiable in time. 
For a fixed $t\in(0,\infty)$, the assumed smoothness on $X$, together with the identification $t= \sz k$ and Taylor's formula in a normed vector space (e.g., \citep[Theorem 7.9-1]{ciarlet_fa}) implies:
\begin{align}
\frac{x_{k+1}-x_k}{\sz}&=\dot{X}(t)+\half \ddot{X}(t)\sz+\littleo{\sz}	\quad \text{ as } s\goesto 0,
\nonumber
\\
\frac{x_{k}-x_{k-1}}{\sz}&=\dot{X}(t)-\half \ddot{X}(t)\sz+\littleo{\sz} \quad \text{ as } s\goesto 0,
\nonumber
\\	
\sz \LLinv G'(y_k)&=\sz \LLinv G'(X(t))+\littleo{\sz} \quad \text{ as } s\goesto 0.
\label{G'-taylor}
\end{align}
The last identity follows from the continuity of $G'$, that of $\LLinv$, and \eqref{PAGD-y-upd}, from which we can deduce $y_k\goesto X(t)$ as $s\goesto 0$.
Plugging \eqref{PAGD-y-upd} into \eqref{PAGD-x-upd} and dividing by $\sz$, we have
$ %\begin{equation}
\frac{x_{k+1}-x_k}{\sz}-	\lambda\frac{x_{k}-x_{k-1}}{\sz}+\sz \LLinv G'(y_k)=0
%\label{AGD-CW-combined}.
$. %\end{equation}
Substituting the above Taylor expansions, and then rearranging, we arrive at
\begin{align}
\half(1+\lambda) \ddot{X}(t) +\frac{1-\lambda}{\sz}\dot{X}(t)+ \LLinv G'(X(t))+\littleo{1}=0 \quad \text{ as } s\goesto 0.
\label{ODE-ansatz}
\end{align}

To make this estimate consistent, interpret $\lambda$ as a function of $s$ and further assume that $(1-\lambda)/\sz\goesto 2\eta$ as $\sz\goesto 0$ for some $\eta\in(0,\infty)$, which yields

\begin{equation}
\ddot{X}(t)+2\eta\dot{X}(t)+\LLinv G'(X(t))=0.
\label{the-ODE}
\end{equation}

\subsection{Derivation of the initial conditions}
\label{sub:ICs}

The initialization $y_0 = x_0$ and \eqref{PAGD-x-upd} with $k=0$ imply
\[
\frac{x_1-x_0}{\sz}=\sz \LLinv G'(x_0)	.
\]
Take the limit $s\goesto0$ and conclude $\dot{X}(0)=0$ since $G'$ and $\dot X$ are assumed to be continuous. Therefore, we arrive at the desired IVP \eqref{the-IVP}.

\begin{rmk}[momentum method] 
	\label{rmk:MM}
	A similar procedure can be carried out far more easily for the so-called \emph{momentum method} (MM). To see this, we recall that
	\[
	\ddot X(t)\approx\frac{x_{k+1}-2x_k+x_{k-1}}{s},
	\quad
	\dot X(t)\approx\frac{x_k-x_{k-1}}{\sz},
	\quad
	G'(X(t))\approx G'(x_k) .
	\]
	Then, the discrete version of the ODE \eqref{the-IVP} becomes
	\[
	x_{k+1}=x_k-sG'(x_k)+(1-2\eta \sz)(x_k-x_{k-1}),
	\]
	which is MM with the weight $1-2\eta \sz$; see \citep[p. 12 (9)]{polyak1964}. This weight is close to $\lambda$:
	\[
	\lambda=\frac{1-\eta\sz}{1+\eta\sz}=1-\frac{2\eta\sz}{1+\eta\sz}\approx 1-2\eta\sz.
	\] 	
	In this sense, MM seems more natural and amenable for analysis than AGD.
	\ermk\end{rmk}

The limiting behavior of MM can also be explained by the IVP \eqref{the-IVP}. Observe that the only essential difference between MM and PAGD is where $G'$ is evaluated, that is, $x_k$ and $y_k$ respectively. And in the limit $s\goesto0$, $x_k$ and $y_k$ are not distinguishable in this setting. However, PAGD exhibits less oscillation than MM since evaluating $G'$ at $y_k$ serves as ``foreseeing'' the uphill of the objective functional, if exists, along the trajectory and ``steering'' to avoid unnecessary oscillating behaviors. 	Recently, a higher order Taylor expansion turns out to help differentiate their performaces (see \citep{Shi2018}).

\section{PAGD as a discretization of the IVP}
\label{subsec:discretization}
Let us label the step size $\sz$, rather than $s$, in order to make the setting more in line with the PAGD algorithm.
Again, it is helpful to have in mind the correspondence: time $t\longleftrightarrow k\sz$ ($k=0,1,2,\cdots$) and position $X(t)\longleftrightarrow x_k$.  First, we will see $y_k$ corresponds to a ``drifted" position without the potential landscape over $[t,t+\sz]$. This can be modeled by $\ddot X(t)+2\eta \dot X(t)=0$, which leads to another energy law 
$
\half \norm{\dot X(t+\sz)}{\LL}^2	= \half \norm{\dot X(t)}{\LL}^2 -2\eta\int_t^{t+\sz}\norm{\dot X(\tau)}{\LL}^2\diff \tau.
$ 
Approximate the speed in the integrand by the average $\frac12(\norm{\dot X(t+\sz)}{\LL}+\norm{\dot X(t)}{\LL})$, then after a short calculation, one obtains $\norm{\dot X(t+\sz)}{\LL}=\lambda\norm{\dot X(t)}{\LL}$. Since the dynamics takes place in a single direction, this implies $\dot X(t+\sz)=\lambda\dot X(t)$. The approximations $\dot X(t)\approx\frac{x_k-x_{k-1}}{\sz}$ and $\dot X(t+\sz)=\frac{y_k-x_k}{\sz}$ lead us to \eqref{PAGD-y-upd}.

Next, we discretize the vector $V(t)$. Since we do not know the minimizer in practice, we remove it from the definition of $v_k$ and discretize $V(t)+x^*=X(t)+\frac{1}{\eta}\dot X(t)$. The approximations $X(t)\approx y_k$ and $\dot X(t)\approx \frac{y_k-x_k}{\sz}$ suggest 
\begin{equation}
	v_k = y_{k}	+\frac{1}{\theta}(y_k-x_{k}),
	\label{xyv-rel2}
\end{equation} 
which leads to the definition of $\{v_k\}_{k\geq1}$ \eqref{PAGD-v-upd} upon combining with the definition of $\{y_k\}$.

Finally, to get the main iterates, $\{x_{k}\}_{k\geq1}$, we discretize \eqref{theV-ODE} using the approximations $\dot V (t) \approx \frac{v_{k+1}-v_k}{\sz}$, $\dot X (t) \approx \frac{y_{k}-x_k}{\sz}$, and the evaluation of $G'$ at $y_k$, then it follows
$
\eta \frac{v_{k+1}-v_k}{\sz}+\eta\frac{y_k-x_k}{\sz}+\LLinv G'(y_k)=0
$.
Plugging in \eqref{PAGD-v-upd} and \eqref{xyv-rel2}, one obtains \eqref{PAGD-x-upd}, the definition of $\{x_k\}_{k\geq1}$.

\section{Literature comparison}
\label{sec:FancyTable}

We summarize our discussion on the existing literature works, and contrast them with our contributions, in Table \ref{tab:literature}.

\begin{center}
\begin{table}[!ht]
	\begin{tabular}{|c|c|c|c|c|c|c|c|}
		\hline
		Ref.& $s=0$ & $s>0$ & $L$&$\mathfrak{B}$ & $\mathcal L$ &Numerics &% Analysis
% 		\begin{tabular}[c]{@{}c@{}}Continuous\end{tabular} &
% 		\begin{tabular}[c]{@{}c@{}}Discrete\end{tabular} &
% 		Lipschitz &
% 		\begin{tabular}[c]{@{}c@{}}Invariant\\Set\end{tabular} &
% 		$\mathcal L$ &
% 		\begin{tabular}[c]{@{}c@{}}Numerics\end{tabular} &
		\begin{tabular}[c]{@{}c@{}}Numerical\\Analysis\end{tabular} 
\\ \hline
		\cite{Wilson2018} & Opt. & Opt. & Glob. & 
		\cellcolor[HTML]{FFCE93} $\times$ &
		\cellcolor[HTML]{FFCE93}\begin{tabular}[c]{@{}c@{}}Bregman\end{tabular} &
		$\times$ &
		\cellcolor[HTML]{FFCE93} $\times$
\\ \hline
		\cite{Attouch2000} &
		\cellcolor[HTML]{CBCEFB}$\times$ &
		\cellcolor[HTML]{CBCEFB}$\times$ &
		\cellcolor[HTML]{CBCEFB}Loc. &
		\cellcolor[HTML]{FFCE93}$\times$ &
		\cellcolor[HTML]{FFCE93}$\times$ &
		\checkmark &
		\cellcolor[HTML]{FFCE93}$\times$
\\ \hline
		\cite{Laborde2020} & Opt. & Opt. & Glob. &
		\cellcolor[HTML]{FFCE93}$\times$ &
		\cellcolor[HTML]{FFCE93}$\times$ &
		$\times$ &
		\cellcolor[HTML]{FFCE93}$\times$
\\ \hline
		\cite{Schaeffer2016} & $\times$ & $\times$ & Glob. & 
		\cellcolor[HTML]{FFCE93}$\times$ &
		\cellcolor[HTML]{FFCE93}$\times$ &
		\checkmark &
% 		\begin{tabular}[c]{@{}c@{}}Difficult PDE \\ (Bellman, \\ Monge-Ampere)\end{tabular} &
		\cellcolor[HTML]{FFCE93}$\times$
\\ \hline
		\cite{Goudou2009} & 
		\cellcolor[HTML]{CBCEFB}$\times$ &
		\cellcolor[HTML]{CBCEFB}$\times$ &
		\cellcolor[HTML]{CBCEFB}Loc. &
		\cellcolor[HTML]{FFCE93}$\times$ &
		\cellcolor[HTML]{FFCE93}$\times$ &
		$\times$ &
		\cellcolor[HTML]{FFCE93}$\times$
\\ \hline
		\cite{Calder2019}&%Calder et al. (2019) &
		Sub. & $\times$ & Glob. &
		\cellcolor[HTML]{FFCE93}$\times$ &
		\cellcolor[HTML]{FFCE93}$\times$ &
		\checkmark &
% 		\begin{tabular}[c]{@{}c@{}}Obstacle \\ mimimal \\ surface\end{tabular} &
		\cellcolor[HTML]{FFCE93}$\times$
\\ \hline
		\cite{Luo2019}&%Luo et al. (2019) &
		Opt. &
		Opt. &
		Glob. &
		\cellcolor[HTML]{FFCE93}$\times$ &
		\cellcolor[HTML]{FFCE93}$\times$ &
		$\times$ &
		\cellcolor[HTML]{FFCE93}$\times$
\\ \hline
		\cite{Siegel2019}&%Siegel (2019) &
		Opt. &
		Opt. &
		Glob. &
		\cellcolor[HTML]{FFCE93}$\times$ &
		\cellcolor[HTML]{FFCE93}$\times$ &
		$\times$ &
		\cellcolor[HTML]{FFCE93}$\times$
\\ \hline
		\cite{Shi2018}&%Shi et al. (2018) &
		Sub. &
		Sub. &
		Glob. &
		\cellcolor[HTML]{FFCE93}$\times$ &
		\cellcolor[HTML]{FFCE93}$\times$ &
		\checkmark &
		\cellcolor[HTML]{FFCE93}$\times$
\\ \hline
		Ours &
		\cellcolor[HTML]{CBCEFB}Opt. &
		\cellcolor[HTML]{CBCEFB}Opt. &
		\cellcolor[HTML]{CBCEFB}Loc. &
		\cellcolor[HTML]{FFCE93}\checkmark &
		\cellcolor[HTML]{FFCE93}\checkmark &
		\checkmark &
% 		\begin{tabular}[c]{@{}c@{}}Nonlinear, \\ fractional \\ Laplacian\end{tabular} &
		\cellcolor[HTML]{FFCE93}\checkmark
\\ \hline
	\end{tabular}
\caption{A comparison of recent works from a numerical PDE point of view. All works that provide convergence rates either in the continuous (column $s=0$) or discrete (column $s>0$) case assume the global Lipschitz condition (column $L$). No work addresses invariant sets (column $\mathfrak B$), incorporates preconditioning explicitly (column $\mathcal L$), nor it explains how concrete numerical examples fit the abstract theory (column Numerical Analysis).}
\label{tab:literature}
\end{table}
\end{center}

\end{document}